\documentclass[11pt]{amsart}

\usepackage[T1]{fontenc}

\usepackage{amsmath,amsthm,amssymb}
\usepackage{graphicx}
\usepackage[all]{xy}
\usepackage{psfrag}

\usepackage{color}
\usepackage{epsfig}
\usepackage{latexsym}

\newtheorem{thm}{Theorem}[section]
\newtheorem{thmintro}{Theorem}
\newtheorem{lem}[thm]{Lemma}
\newtheorem*{acknowledgement}{Acknowledgements}

\newcommand{\plc}{\mathbb{C} \mathbb{P}^2}
\newcommand{\drc}{\mathbb{C} \mathbb{P}^1}
\newcommand{\plr}{\mathbb{R} \mathbb{P}^2}
\newcommand{\drr}{\mathbb{R} \mathbb{P}^1}
\newcommand{\cc}{\mathbb{C}}
\newcommand{\rr}{\mathbb{R}}
\newcommand{\zz}{\mathbb{Z}}
\newcommand{\TT}{\mathbb{T}}
\newcommand{\TorCl}{T_{\Clif}}
\newcommand{\TetCS}{\Theta_{\CS}}
\newcommand{\Tetsurg}{\Theta_{\surg}}
\newcommand{\tTetCh}{\widetilde{\Theta}_{\Ch}}
\newcommand{\ctg}{T^{\ast}}
\newcommand{\cerc}{\mathbb{S}^1}

\newcommand{\rCh}{\rho_{\Ch}}

\DeclareMathOperator{\Clif}{Cliff}
\DeclareMathOperator{\Ch}{Ch}
\DeclareMathOperator{\CS}{CS}
\DeclareMathOperator{\surg}{surg}

\title[Toric constructions of monotone Lagrangian submanifolds]{Toric constructions of monotone Lagrangian submanifolds in~$\plc$ and~$\drc \times \drc$}
\date{}
\author{Miguel Abreu}
\address[MA]{Center for Mathematical Analysis, Geometry and Dynamical Systems,  
Instituto Superior T\'ecnico, Universidade de Lisboa\\Av. Rovisco Pais, 1049-001 Lisboa, Portugal}
\thanks{MA was partially funded by FCT/Portugal through projects UID/MAT/04459/2013 
and EXCL/MAT-GEO/0222/2012.}
\email{mabreu@math.tecnico.ulisboa.pt}

\author{Agn\`es Gadbled}
\address[AG]{Centro de Matem\'atica da Universidade do Porto, Departamento de Matem\'atica da FCUP\\
Rua do Campo Alegre, 687, 4169-007 Porto, Portugal}
\thanks{AG was partially supported by CMUP (UID/MAT/00144/2013), which is funded by FCT (Portugal) with national (MEC) and European structural funds through the programs FEDER, under the partnership agreement PT2020.}
\email{agnes.gadbled@fc.up.pt}
\thanks{MA and AG were both partially funded through the project PTDC/MAT/117762/2010. 
The present work is part of the authors' activities within CAST, a Research Networking Program of the 
European Science Foundation.}

\begin{document}

%\title[Toric constructions of monotone Lagrangian submanifolds]{Toric constructions of monotone Lagrangian submanifolds in~$\plc$ and~$\drc \times \drc$}
%\author[Miguel Abreu and Agn\`es Gadbled]{Miguel Abreu and Agn\`es Gadbled
%\footnote{
%MA was partially funded by FCT/Portugal through projects UID/MAT/04459/2013 
%and EXCL/MAT-GEO/0222/2012. AG was partially supported by CMUP (UID/MAT/00144/2013), which is funded by FCT (Portugal) with national (MEC) and European structural funds through the programs FEDER, under the partnership agreement PT2020. MA and AG were both partially funded through the project PTDC/MAT/117762/2010. 
%The present work is part of the authors' activities within CAST, a Research Networking Program of the 
%European Science Foundation.}
%}

\maketitle

\begin{abstract}
We extract from a toric model of the Chekanov-Schlenk exotic torus in~$\plc$ 
methods for constructing Lagrangian submanifolds in toric symplectic manifolds. 
These constructions allow for some control of monotonicity. We 
recover this way some known monotone Lagrangians in the toric symplectic 
manifolds~$\plc$ and~$\drc \times \drc$ as well as new examples.
\end{abstract}

%\maketitle

\section{Introduction}

One can reconstruct~$\plr$, the real part of~$\plc$ with respect to the standard 
conjugation map, from its image under the standard moment map of~$\plc$. 
The real part projects under the moment map of~$\plc$ onto the entire 
moment polytope, each point in the interior of the polytope having four preimages, 
the points on the interior of the edges on the boundary of the triangle 
having two preimages, the points on the vertices having one preimage. 
By taking four copies of the moment polytope and gluing them along the edges 
according to the prescriptions of the torus action on~$\plc$ (see~\cite{MR3042606}), 
one recovers the real projective plane.\\

From the Lagrangian submanifold point of view, $L = \plr$ is an important example of 
monotone Lagrangian submanifold in~$\plc$. Monotone means that there exists 
a positive constant~$K_L$ such that
$$\forall u \in H_2(\plc, L), \int_u \omega = K_L \,  \mu_L(u),$$
where~$\mu_L : H_2(\plc, L) \rightarrow \zz$ is the Maslov class of~$L$.

The exotic torus of Chekanov and Schlenk (see~\cite{MR2735030}) is another important 
example of monotone Lagrangian submanifold of~$\plc$. 
The second author proved in~\cite{Exotic} that this torus is Hamiltonian isotopic 
to a torus described by Biran and Cornea in~\cite{Bir-Cor-uniruling}. To do so, 
she Hamiltonian-isotoped  both tori to a so-called modified Chekanov torus~$\tTetCh$.
This torus has a nice image under the moment map and can be reconstructed, 
as the real projective space, out of copies of this image and gluing patterns.
The rules for gluing are of two types. The first are coming from the definition 
of the moment map and are the same as the ones used for the real part. 
The second are new and we have managed to interpret them as Lagrangian 
surgeries of two copies of the real part intersecting transversely 
at one isolated point and cleanly (in the sense of 
Pozniak~\cite{MR1736217}) along a circle.

The surgery for two Lagrangian submanifolds intersecting transversely at 
a point has been developped by Polterovich in~\cite{MR1097259} and we have 
modified it to keep a toric description of the result of the surgery.
The surgery for two Lagrangian submanifolds intersecting along an isotropic 
submanifold not reduced to a point is new and we intend to develop it in full generality 
in a future work. We show in our case:

\begin{thmintro}
The Chekanov--Schlenk torus is Hamiltonian isotopic in~$\mathbb{C} \mathbb{P}^2$ 
to a Lagrangian torus obtained from two copies of~$\mathbb{R} \mathbb{P}^2$ 
by Lagrangian surgeries at a point and along an isotropic circle.
\end{thmintro}

With our method we can also recover Lagrangian embeddings of some surfaces in $\rr^4$
that were constructed by Givental in~\cite{MR868559} and that we then embed 
in~$\mathbb{C} \mathbb{P}^2$ or $\mathbb{C} \mathbb{P}^1 \times \mathbb{C} \mathbb{P}^1$.
The advantage of this construction is that we have a good control of the monotonicity condition.
We knew so far only the monotone Lagrangian embeddings of tori 
in~$\mathbb{C} \mathbb{P}^1 \times \mathbb{C} \mathbb{P}^1$ and of tori 
and real projective planes in~$\mathbb{C} \mathbb{P}^2$. Our method enables 
us to prove

\begin{thmintro}
There exists a monotone Lagrangian embedding of the connected sum of a surface of 
genus~$2$ and a Klein bottle in~$\mathbb{C} \mathbb{P}^2$ and a monotone Lagrangian embedding 
of the connected sum of a surface of genus~$4$ and a Klein bottle in the
product~$\mathbb{C} \mathbb{P}^1 \times \mathbb{C} \mathbb{P}^1$.
\end{thmintro}

Note that neither the Klein bottle (see~\cite{MR2583968}) nor the orientable surface 
of genus~$2$ (see~\cite{MR2091366}) can be embedded as Lagrangian submanifolds of~$\plc$.

The structure of this article is as follows: in Section~\ref{sec:LocalModels} we 
study the gluing patterns for the modified Chekanov torus and we describe the surgeries 
we will use; in Section~\ref{sec:Constructionwithsurgeryatapoint} 
we give our main construction and study the 
monotonicity of examples we can get with the surgery at a point; finally, in 
Section~\ref{sec:Surgeryinbundle}, we use the surgery for an intersection along a circle 
to describe two monotone Lagrangian embeddings.

\begin{acknowledgement}
We thank Leonardo Macarini for useful conversations that, in particular, led to the 
identification of a mistake in an earlier draft of this paper. 
We also thank an anonymous referee for a very careful reading and detailed list of corrections.
\end{acknowledgement}

\section{The local models}
\label{sec:LocalModels}

\subsection{A toric model of the exotic torus of Chekanov and Schlenk in~$\plc$}

There is a well-known monotone torus in~$\plc$ called the Clifford torus which can be 
described in homogeneous coordinates as
$$\TorCl = \left\{{ \left. \left[ { e^{i\alpha}:  e^{i\beta}: 1  }\right] \right| 
\alpha, \beta \in [0, 2 \pi] } \right\}.$$

Its image under the moment map of~$\plc$ (corresponding to the normalization of the 
symplectic form we use in Section~\ref{sec:Constructionwithsurgeryatapoint})
$$\begin{array}{cccc}
 \mu: & \plc & \longrightarrow & \rr^2  \\
    & [z_0:z_1:z_2] & \longmapsto &  \left( {3 \frac{|z_0|^2}{\sum |z_i|^2}, 3 \frac{|z_1|^2}{\sum |z_i|^2} } \right),
\end{array}$$
is the barycenter~$(1,1)$ of the image of~$\plc$, the triangle obtained as the 
convex hull of the points~$(0,0),(3,0),(0,3)$.

In 2004, Chekanov and Schlenk (see~\cite{MR2735030}) have studied the torus given  
in homogeneous coordinates of~$\mathbb{C} \mathbb{P}^2$:
$$\TetCS = \left\{{ \left. \left[
{\frac{1}{\sqrt{2}}\gamma(s) e^{i\theta}: \frac{1}{\sqrt{2}}\gamma(s) e^{-i\theta}: 
\sqrt{\frac{3}{\pi} - \mid \gamma(s) \mid^2} 
}\right] \right| \theta, s \in [0, 2 \pi] } \right\}$$
where $\gamma : [0, 2 \pi] \longrightarrow \cc$ parametrizes a curve enclosing a 
domain of area~$1$ lying in the disk centered in the origin and of 
area~$2 + \varepsilon$ of~$\cc$, in the half-disk of complex numbers of positive 
real part (see Figure~\ref{fig:gamma}).
They have proved (see~\cite{MR2735030,Chek-SchI,Chek-SchII})
that~$\TetCS$ is a monotone Lagrangian torus in~$\plc$, non-displaceable and 
non-Hamiltonian isotopic to the Clifford torus (therefore called exotic) in~$\plc$.

By its definition, the Chekanov--Schlenk torus projects under the moment map~$\mu$ 
to a segment lying in the diagonal line of~$\rr^2$. More precisely the image is 
$$\left\{{ (x,x) \in \rr^2 \left| { \, x \in \left[{\frac{\pi}{2}\rho_{\min}^2, \frac{\pi}{2}\rho_{\max}^2} \right] } \right. }\right\}$$
where  $\rho_{\min}$ is the minimum of $|\gamma(s)|$ and $\rho_{\max}$ is the maximum.

There is a description of this exotic torus more adapted to the toric picture, 
that enables to reconstruct the torus from its moment map image as in 
the case of the real part of~$\plc$.

Such a description can be obtained by considering the modified Chekanov torus 
of~\cite{Exotic}. 
This torus is a torus Hamiltonian isotopic to the exotic torus of Chekanov--Schleck 
and is defined in homogeneous coordinates (with the normalizations 
of~\cite{MR2735030}) by
$$\tTetCh = \left \{ \left. \left[{\cos(\theta) \gamma(s): 
\sin(\theta) \gamma(s) : \sqrt{\frac{3}{\pi} - \mid \gamma(s) \mid^2} }\right] 
\; \right| \; \theta, s \in [0, 2 \pi] \right\}.$$

The image of the torus under the moment map~$\mu$ can be parametrised by 
$$\mu(\tTetCh) = \left\{ {\left( { \pi \cos^2(\theta) |\gamma(s)|^2, \pi  \sin^2(\theta) |\gamma(s)|^2 }\right) | \; \theta, s \in [0, 2 \pi] }\right\}.$$
It is a trapezoid sitting inside the polytope of $\plc$ between the two parallel lines 
$x + y = \pi \rho_{\min}^2$ and $x + y = \pi \rho_{\max}^2$.

If the curve $\gamma$ is such that $\gamma(0)=\rho_{\min}$, $\gamma(\pi)=\rho_{\max}$, 
$\gamma$ is symmetric with respect to the real axis, 
and each point $|\gamma(s)|=\rho(s)$ has only one preimage $s \in (0,\pi)$ 
(see for example Figure \ref{fig:gamma}), then 
for $s \neq 0, \pi$ and $\theta \neq 0,\pi$, the point
\begin{equation}
\label{eq:imagetorus}
\left( { \pi \cos^2(\theta) |\gamma(s)|^2, \pi \sin^2(\theta) |\gamma(s)|^2 }\right)
\end{equation}
has $8$ preimages in the torus.\\

Recall (see for example~\cite{MR2091310,MR1853077}) that the image of the moment map for~$\plc$, 
or for a general (compact, connected) toric manifold~$(M,\omega)$ is a convex polytope~$P$ such that the 
fiber of each point of~$P$ is an isotropic torus. 
Recall also that we have action-angle coordinates on the preimage~$\mathring{M}$ of the 
interior~$\mathring{P}$ which is the open dense set in~$M$ consisting of all the points where 
the action of the torus~$\TT^n$ is free. One can describe this set as
$$\mathring{M} \cong \mathring{P} \times \TT^n 
=  \left\{ {\left( {  x_1, \ldots, x_n , e^{i\theta_1}, \ldots, e^{i\theta_n}  }\right) 
\mid x  \in \mathring{P}, \theta \in \rr^n / 2 \pi \zz^n  }\right\},$$
where~$(x,\theta)$ are the action-angle coordinates for the symplectic form
$$\omega = \sum d x_j \wedge d \theta_j.$$
 
In the case of~$\tTetCh$, we parametrise the curve $\gamma$ above the real axis by 
$$\gamma(s) = \rho(s) e^{i t(s)},$$ 
such that $t(s) \in [0, t_{\max}], t_{\max} < \frac{\pi}{2}, t(0) = 0, t(\pi) = 0$.

\begin{figure}[htbp]
  \begin{center}
   \psfrag{C}{$\cc$}
   \psfrag{1}{$\scriptstyle{1}$}
   \psfrag{0}{$\scriptstyle{0}$}
   \psfrag{gamma}{$\scriptstyle{\gamma}$}
   \psfrag{rhomin}{$\scriptstyle{\rho_{\min}}$}
   \psfrag{rhomax}{$\scriptstyle{\rho_{\max}}$}
   \psfrag{tmax}{$\scriptstyle{t_{\max}}$}
   \includegraphics[height=7cm]{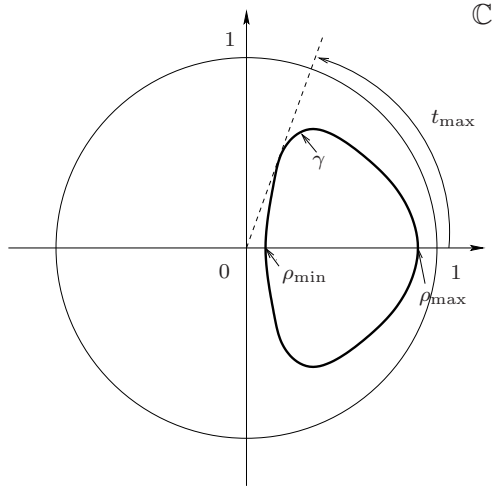}
   \caption{A suitable curve $\gamma$} 
   \label{fig:gamma}
  \end{center}
\end{figure} 

Then for a fixed $s$ in $(0,\pi)$ and a fixed $\theta$ in $(0, \frac{\pi}{2})$, 
the eight preimages of the point (\ref{eq:imagetorus})
in~$\mu(\tTetCh)$ are given in action-angle coordinates by:
$$A_{\epsilon,k,\ell} = \left( { \pi \cos^2(\theta) |\gamma(s)|^2, \pi \sin^2(\theta) |\gamma(s)|^2 , \epsilon t(s)+\frac{\pi}{2} + k \pi, \epsilon t(s)+ \frac{\pi}{2} + \ell \pi } \right)$$
with $\epsilon \in \{-1, 1\}$, $k, \ell \in \{0, 1\}$.

When $s$ goes to $0$ or $\pi$, $t(s)$ goes to $0$ and the points in the 
torus fiber converge (moving along the diagonal direction) 
towards one of the four points 
$$\left( {\frac{\pi}{2} + k \pi,  \frac{\pi}{2} + \ell \pi }\right), k, \ell \in \{0, 1\},$$
see Figure~\ref{fig:gluing}.

\begin{figure}[htbp]
  \begin{center}
   \psfrag{0}{$\scriptstyle{0}$}
   \psfrag{pi/2}{$\scriptstyle{\pi/2}$}
   \psfrag{pi}{$\scriptstyle{\pi}$}   
   \psfrag{3pi/2}{$\scriptstyle{3\pi/2}$}
   \psfrag{2pi}{$\scriptstyle{2\pi}$}
   \psfrag{$A_{1,0,0}$}{$A_{1,0,0}$}
   \psfrag{$A_{1,1,0}$}{$A_{1,1,0}$}
   \psfrag{$A_{1,0,1}$}{$A_{1,0,1}$}
   \psfrag{$A_{1,1,1}$}{$A_{1,1,1}$}
   \psfrag{$A_{-1,0,0}$}{$A_{-1,0,0}$}
   \includegraphics[height=12cm]{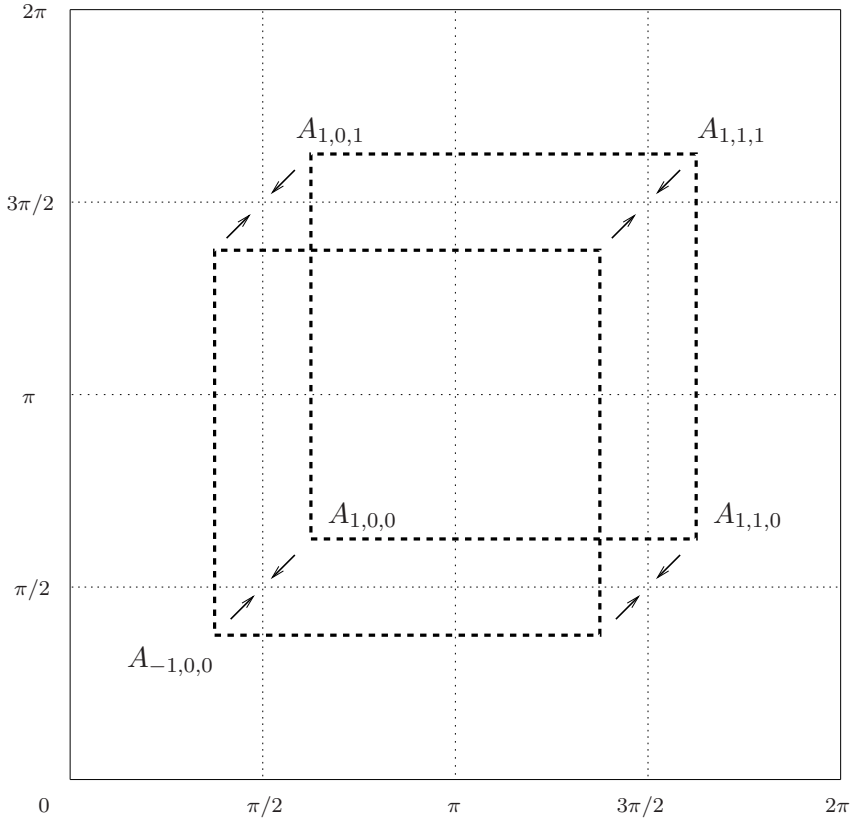}
   \caption{The points are at the four corners of each of the 
   dashed squares; 
   when $s$ goes to $0$, the corners of the two squares are identified} 
   \label{fig:gluing}
  \end{center}
\end{figure}

This describes the gluing of the trapezoid along the segments 
$x + y = \pi \rho_{\min}^2$ and $x + y = \pi \rho_{\max}^2$.\\

\subsection{The interpretation of the gluing along the segment $x + y = \pi \rho_{\min}^2$}
\label{sec:surgeryatapoint}

The gluing along the segment $x + y = \pi \rho_{\min}^2$ can be described as the 
Lagrangian surgery defined in~\cite{MR1097259} of two Hamiltonian isotopic copies 
of the real part~$\plr$ intersecting transversaly at the origin~$[0:0:1]$ 
(see Section~\ref{sec:exotictorus} for the details).

Let us describe the Lagrangian surgery we will use in the rest of this article 
which is a slight modification of~\cite{MR1097259}. 
Following Polterovich, one does the surgery of two transverse Lagrangians  
in a local chart around an intersection point. 
In this chart, one finds an almost-complex structure $j$ such that locally 
 $l_1 = j l_0$ and the local Lagrangian handle 
is the image of the sphere in $l_0$ times $[-T,T]$, for $T$ large under the map
$(\xi , t) \mapsto e^{-t} \xi + e^{t} j \xi$. One then connects the handle on its boundaries 
to the original Lagrangians by some smoothing.

Because we aim to control the monotonicity condition and keep the standard local chart 
in~$\plc$, we explicitely describe the Lagrangian surgery we will be using, 
without the use of an auxiliairy almost-complex structure $j$ in $\cc^2$.

\subsubsection*{The handle between two Lagrangian linear subspaces of $\cc^2$}

Consider the linear~$\cc^2$ and the two Lagrangian linear subspaces 
$$l_0 = \rr^2$$ 
and 
$$l_1 = \left( { \begin{array}{cc} e^{i \alpha} & 0 \\0 & e^{i \beta} \end{array} } \right) \rr^2$$ the image of~$l_0$ by the Hamiltonian diffeomorphism defined by the diagonal 
matrix~$diag (e^{i \alpha}, e^{i \beta})$ for~$\alpha$ and~$\beta$ not a 
multiple of~$\pi$.

We define a handle~$h$ parametrised by:
$$h = \left \{ { \left. e^{-t} \left( { \begin{array}{c} x_0\\x_1 \end{array} } \right) \\
+ e^{t} \left( { \begin{array}{c} e^{i \alpha} x_0\\ e^{i \beta} x_1 \end{array} } \right) \\
\right|   \begin{array}{c} t \in \rr \\x_0^2+x_1^2 = 1 \end{array} } \right \}.$$
It is asymptotic to~$l_0$ when t goes to~$- \infty$ and to~$l_1$ when t goes to~$+ \infty$.
One checks that it is a Lagrangian handle when~$\sin (\alpha) = \sin (\beta)$.

Note that for the same reason, the handle
$$h' = \left \{ { \left. e^{-t} \left( { \begin{array}{c} x_0\\x_1 \end{array} } \right) \\
+ e^{t} \left( { \begin{array}{c} e^{i \alpha} (-x_0) \\ e^{i \beta} (-x_1) \end{array} } \right) \\
\right|   \begin{array}{c} t \in \rr \\x_0^2+x_1^2 = 1 \end{array} } \right \}$$
is also a Lagrangian submanifold asymptotic to~$l_0$ and~$l_1$ and corresponds to the first 
handle for the angles~$(\alpha + \pi, \beta + \pi)$. \\

Note that when~$\sin (\alpha) = - \sin (\beta)$, one can also define two 
Lagrangian handles parametrized by
$$h = \left \{ { \left. e^{-t} \left( { \begin{array}{c} x_0\\x_1 \end{array} } \right) \\
+ e^{t} \left( { \begin{array}{c} e^{i \alpha} x_0\\ e^{i \beta} (-x_1) \end{array} } \right) \\
\right|   \begin{array}{c} t \in \rr \\x_0^2+x_1^2 = 1 \end{array} } \right \}$$
and
$$h' = \left \{ { \left. e^{-t} \left( { \begin{array}{c} x_0\\x_1 \end{array} } \right) \\
+ e^{t} \left( { \begin{array}{c} e^{i \alpha} (-x_0) \\ e^{i \beta} x_1 \end{array} } \right) \\
\right|   \begin{array}{c} t \in \rr \\x_0^2+x_1^2 = 1 \end{array} } \right \}.$$

\subsubsection*{The smoothing}

For some (large)~$T$,  denote by~$h_T$ the image of the handle for~$t$ between~$-T$ and~$T$:
$$h_T = \left \{ { \left. e^{-t} \left( { \begin{array}{c} x_0\\x_1 \end{array} } \right) + e^{t} \left( { \begin{array}{c} e^{i \alpha} x_0\\ e^{i \beta} x_1 \end{array} } \right) \right|   
\begin{array}{c} t \in [-T,T] \\x_0^2+x_1^2 = 1 \end{array} } \right \}.$$

Fix a parameter~$T$ large. We smooth the handle at the ends of~$h_T$ as in~\cite{MR1097259}.
Notice that the original surgery of Polterovich corresponds to the case  
when~$\alpha = \beta = \frac{\pi}{2}$ and we can obtain any of our handles from Polterovich's 
surgery by applying the linear transformation of~$\cc^2 = \rr^4$ 
with matrix
$$\left( { \begin{array}{cccc} 
1 & \cos(\alpha) & 0 & 0 \\
0 & \sin(\alpha) & 0 & 0 \\
0 & 0 & 1 & \cos(\beta) \\
0 & 0 & 0 & \sin(\beta) \end{array} } \right).$$
Hence we can smooth the handle at the boundary of~$h_T$ when~$\alpha = \beta = \frac{\pi}{2}$ as 
in~\cite{MR1097259} and then take the image of this smoothing by the linear map above to get a 
smoothing in our case.

Note that this surgery lies inside a big ball of radius~$R$, and outside it,
the Lagrangian submanifold obtained is the union of~$l_0$ and~$l_1$. 
As the linear Lagrangians are homogenous with respect to homotheties centered at the origin 
of~$\cc^2$, one can use a conformal transformation to make this Lagrangian surgery happen 
in a ball~$B_0$ of small radius. 
Equivalently, one can take the handle to be, not the image of the unit sphere in~$l_0$, 
but of a smaller one
$$x_0^2 + x_1 ^2 = \varepsilon_1^2$$
and the smoothing happening outside a ball of radius~$\varepsilon_2$ such that the Lagrangian 
submanifold after surgery identifies with the union of~$l_0$ and~$l_1$ outside a ball of 
radius~$\varepsilon$ for a parameter~$\varepsilon > \varepsilon_2 > \varepsilon_1 > 0$ small.

\subsubsection*{Controling the area}

Let us study the restriction of this surgery (before conformal transformation) 
to one factor~$\cc$ of~$\cc^2$, for instance the first. 
The trace of the surgery along this coordinate is given by:
$$h_T^1 = \left \{ { \left. e^{-t} + e^{t}  e^{i \alpha} \right|   
 t \in [-T,T]  } \right \}$$
for some large~$T$, followed by some small smoothing between its ends 
and the original~$l_0$ and~$l_1$, together with the symmetric curve about the origin.

Let us compute the area between the original~$l_0$ and~$l_1$ and one of the 
arcs of the surgery, 
namely the area in grey in Figure~\ref{fig:controlarea}.

\begin{figure}[htbp]
  \begin{center}
   \psfrag{L0}{$l_0$}
   \psfrag{L1}{$l_1$}
   \psfrag{L}{$h_T^1$}   
   \includegraphics[height=6cm]{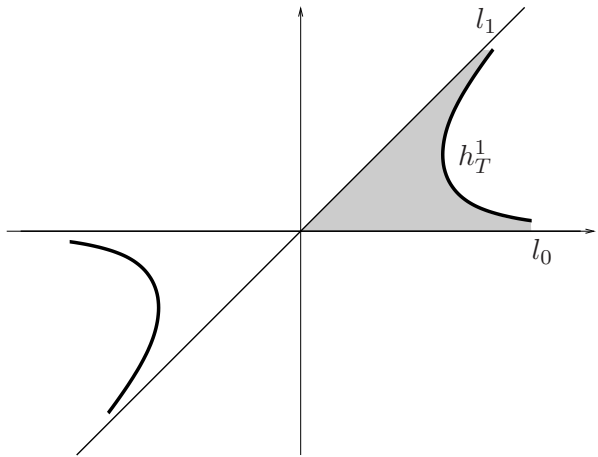}
   \caption{The intersection of the handle with the first $\cc$-factor of $\cc^2$} 
   \label{fig:controlarea}
  \end{center}
\end{figure}

As before, we can deduce this area from the computation of the area when~$\alpha = \frac{\pi}{2}$. In 
this case, $h_T^1$ is the arc of hyperbola in the plane 
given by the equation~$xy = 1$ so that the area in grey is equal to~$2 T + 1$.
This area is equivalent to~$2 T$ when~$T$ goes to~$+ \infty$.
Moreover, when $T$ is large, the smoothing between the original linear Lagrangians 
and the arc of hyperbola is very small so that the area is still equivalent to~$2T$.

The other cases can be obtained from the~$\alpha = \frac{\pi}{2}$ situation by applying the 
linear transformation of the plane~$\cc = \rr^2$ given by the matrix
$$\left( { \begin{array}{cc} 1 & \cos(\alpha) \\0 & \sin(\alpha) \end{array} } \right),$$
so that the area considered above is equivalent to~$2 \sin(\alpha) \, T$  when~$T$ goes to~$\infty$.

In particular, when $\sin (\alpha) = \sin (\beta)$, the condition which ensures the surgery 
to be Lagrangian, the areas between the original Lagrangians and the handle along each 
coordinate are equivalently the same. For~$T$ large enough, we can (and will) do the 
smoothing so that the areas along each $\cc$-factor are equal.

Now, given the conformality property of the surgery, we can make this area as small as 
we want and equal to some~$a(\varepsilon)$ (small) if the surgery is done inside the ball 
of radius~$\varepsilon$.

\subsubsection*{In a general symplectic manifold}

Let~$L_0$ and~$L_1$ be two Lagrangian submanifolds of a symplectic manifold~$W$ intersecting 
transversally at a point~$x_0$.
One can take a Darboux chart~$U_0$ around~$x_0$ symplectomorphic to a ball~$B_0$ endowed with 
the standard symplectic form of~$\cc^2$ such that under the Darboux map, the two Lagrangians
are the intersection of Lagrangian linear subspaces of~$\cc^2$ with the ball. One is then 
in the linear situation from above and can perform the surgery in the ball as described provided
the sines of angles between the restriction of the linear Lagrangians to each factor of~$\cc^2$ 
are the same (up to sign).

\subsection{The interpretation of the gluing along the segment $x + y = \pi \rho_{\max}^2$}
\label{sec:isotropicgluing}

As we shall see in Section~\ref{sec:exotictorus}, the gluing along the segment $x + y = \pi \rho_{\max}^2$ can be interpreted as the Lagrangian surgery along a circle of two Hamiltonian 
isotopic copies of the real part~$\plr$ intersecting along a circle in the $\drc$ at infinity
$$\left\{ {[z_0 : z_1 : z_2] \mid z_2 = 0 }\right\}.$$

This circle is isotropic and, as in the case of the surgery at a point, the surgery along
an isotropic submanifold is a local process that we will now describe.

\subsubsection*{The neighbourhood of an isotropic manifold}

Let~$P$ be a symplectic manifold of dimension~$2n$. 
Let~$N$ be an isotropic submanifold of dimension~$k$. 
In~\cite{MR0464312}, Weinstein noticed that
 the tangent bundle of~$P$  along~$N$ is isomorphic as symplectic vector 
bundle over~$N$ to
$$(TN \oplus TN^*) \oplus SN(N,P),$$
where~$SN(N,P)  = TN^{\bot} / TN$ is called the symplectic normal bundle of~$N$ in~$P$.\\

Conversely, one can embed any manifold which is the base of a symplectic vector bundle as 
an isotropic submanifold of a symplectic manifold such that the tangent bundle looks like this:

\begin{thm}[The existence theorem, Weinstein \cite{MR633631}] Let~$N$ be a manifold of 
dimension~$k$ and~$E \rightarrow N$ a symplectic vector bundle with fibre dimension~$2(n-k), k \leq n$. 
Then~$N$ can be embedded as an isotropic submanifold of a symplectic manifold~$P(E)$ of 
dimension~$2n$ such that the tangent bundle of~$P(E)$ along~$N$ is isomorphic as symplectic 
vector bundle to the sum~$(TN \oplus TN^*) \oplus E$.
\end{thm}

This space~$P(E)$ is the Whitney sum~$P(E) = \ctg N \oplus E$ as Weinstein explains 
in~\cite{MR0464312}. 
The symplectic structure on~$P(E)$ is not canonical and is described in~\cite{MR633631}.\\

And we have a uniqueness result:

\begin{thm}[Weinstein \cite{MR0464312}] The isotropic manifold theorem: Let~$N$ be a manifold of dimension~$k$. 
Then the extensions of~$N$ to a~$2n$-dimensional symplectic manifold in which~$N$ is isotropic 
are classified, up to local symplectomorphism about~$N$, by the isomorphism classes of~$2(n-k)$-dimensional 
symplectic vector bundles over~$N$.
\end{thm}

This means that if~$N$ is an isotropic submanifold of a symplectic manifold~$P$, 
then a neighbourhood of~$N$ in~$P$ is symplectomorphic to a neighbourhood of the embedding of~$N$ 
in~$P(E)$ for~$E = SN(N,P)$.

We will extend the surgery at an intersection point of two Lagrangian submanifolds to the 
surgery of two Lagrangian submanifolds intersecting cleanly (in the sense of 
Pozniak~\cite{MR1736217}) along an isotropic submanifold. 
In his thesis, Pozniak proves 
\begin{thm}[Pozniak \cite{MR1736217}] If two Lagrangian submanifolds~$L_0$ and~$L_1$ 
of a symplectic manifold~$P$ intersect cleanly along~$N$, that 
is if~$N = L_0 \cap L_1$ and for each~$x \in N$, $T_x N = T_x L_0 \cap T_x L_1$, then there 
exists a vector bundle~$L \longrightarrow N$ such that a neighbourhood of~$N$ in~$P$ is 
symplectomorphic to a neighbourhood of~$N$ in~$T^*L$, $L_0$ being mapped to the zero section
of~$T^*L$ and~$L_1$ to the conormal of~$N$ in~$T^*L$.
\end{thm}

In this setting, identifying~$L_0$ and~$L_1$ with their image in~$T^*L$, one can see 
that~$E = SN(N,T^*L)$ is isomorphic to the Whitney sum of the vector 
bundles~$L \rightarrow N$ and~$L^* \rightarrow N$ and that in the Whitney sum 
$P(E) = T^*N \oplus E$, the Lagrangian~$L_0$ is mapped to the zero section 
in the~$T^*N$-summand and to~$L \oplus \{0\}$ in the~$E$-summand and the 
Lagrangian~$L_1$ is mapped to the zero section 
in the~$T^*N$-summand and to~$\{0\} \oplus L^*$ in the~$E$-summand, so that 
the intersection of~$L_0$ and~$L_1$ is the sum of the zero-section of~$T^*N$ 
and the transverse intersection in each fibre of~$E$ of~$L_x \oplus \{0\}$ 
with~$\{0\} \oplus L_x^*$.

The surgery we will construct in this neighbourhood will also fiber over~$N$, be equal 
to the zero-section in the~$T^*N$-summand and will resolve the intersection in 
each fiber of~$E$, so that it is enough to define it in the symplectic normal 
bundle~$E$.

\subsubsection*{The bundle surgery}

In this paper, the constructions are done only in real dimension~$4$ with the clean 
intersection of two Lagrangians along an isotropic circle. 
The symplectic normal bundle~$E = SN(N,T^*L)$ is then a rank~$2$ symplectic vector bundle. 
There exists only 
one rank~$2$ symplectic vector bundle over the circle, the trivial bundle $\cerc \times \cc$.
However, the rank $1$ Lagrangian subbundle~$L \rightarrow N$ can be the trivial line 
bundle or the non-orientable line bundle over the circle. 
So in real dimension~$4$ one will be in one of the following case:
\begin{itemize}
\item either~$E$ is~$\cerc \times \cc$ and~$L \rightarrow N$ is~$\cerc \times \rr$, 
the trivial real subbundle;
\item or~$E$ can be described as $[0,1] \times \cc$ identifying the fiber at~$0$ 
and the fiber at~$1$ by multiplication by~$-1$ 
and~$L \rightarrow N$ is the associated subbundle $[0,1] \times \rr$ with the same 
identification. 
\end{itemize}

Now if the restriction of the Lagrangian~$L_1$ to~$E$ is the associated subbundle with 
fiber~$e^{i \alpha} \rr$ as it will be the case in our examples, 
one can perform a surgery in dimension~$1$ parametrized in each fiber as
$$\left \{ { \left. e^{-t} x + e^{t}  e^{i \alpha} x \right|   
 t \in [-T,T], x \in \rr, x^2 = \varepsilon_1^2  } \right \}$$
 or
 $$\left \{ { \left. e^{-t} x + e^{t}  e^{i \alpha} (-x) \right|   
 t \in [-T,T], x \in \rr, x^2 = \varepsilon_1^2  } \right \},$$
 followed by a smoothing at the end.
 
Now note that in this situation, the restrictions of~$L_0$ and~$L_1$ are invariant 
by multiplication by~$-1$, so the handles and the smoothing can be made invariant as well. 
Therefore, in both cases, the trivial and the non-trivial symplectic bundle over~$\cerc$, 
the change of trivialisation preserves the construction so that the handle can be defined 
globally and fibers over~$\cerc$.

\section{Construction of new monotone Lagrangian submanifolds using the surgery at a point}
\label{sec:Constructionwithsurgeryatapoint}

\subsection{A non-orientable monotone Lagrangian in~$\plc$}

In the following sections, we explain how to get via a Lagrangian surgery on two copies 
of~$\plr$ a monotone Lagrangian connected sum of a Klein bottle and an orientable surface of 
genus two in~$\plc$.

\subsubsection{The construction}
\label{sec:KSigma-construction}

We will take two copies of~$\plr$ in~$\plc$ that intersect in three points exactly, 
the three points of~$\plc$ projecting on the three corners of the image of the moment 
map.
Let us consider:
$$L_0 = \{ \left. [x_0:x_1:x_2] \in \plc \right|  x_0,x_1,x_2 \in \rr \}$$
and
$$L_1 = \{ \left. [e^{i \frac{\pi}{3}} x_0: e^{- i \frac{\pi}{3}} x_1:x_2] \in \plc \right|  
x_0,x_1,x_2 \in \rr \}.$$
These are two copies of $\plr$ in $\plc$, $L_1$ being obtained from $L_0$ by a 
Hamiltonian isotopy.
Indeed, $L_1$ is the image of $L_0$ under the map given by the action of the 
following diagonal matrix on the two first coordinates of the homogeneous coordinates:
$$A = \left ( { 
\begin{array}{cc}
e^{i \frac{\pi}{3}} & 0  \\
0 & e^{-i \frac{\pi}{3}} 
\end{array}
} \right).
$$
It is the time-one map of the transformation given by 
$$A_t = \left ( { 
\begin{array}{cc}
e^{i t \frac{\pi}{3}} & 0  \\
0 & e^{-i t \frac{\pi}{3}} 
\end{array}
} \right)
$$
$A_t \in SU(2)$ and it defines a Hamiltonian diffeomorphism $\Phi_t$ of $\plc$.

One can check that these two copies of $\plr$ intersect in the three 
points $[0:0:1]$, $[0:1:0]$, $[1:0:0]$.

We want to perform a Lagrangian surgery at each intersection point 
as described in Section~\ref{sec:surgeryatapoint}. Let us give the choices 
of handles we make for the construction.

At $[0:0:1]$, the local chart is 
$$[z_0:z_1:z_2] \mapsto \left( \frac{z_0}{z_2}, \frac{z_1}{z_2} \right),$$ 
so that locally, $L_0$ is the real plane 
$$l_0 = \{(x_0,x_1) | \, x_0,x_1 \in \rr\}$$ 
and $L_1$ is 
$$l_1 = \{(e^{i \frac{\pi}{3}} x_0, e^{- i \frac{\pi}{3}} x_1) | \, x_0, x_1 \in \rr\}.$$
We are in the case when $\sin (\frac{\pi}{3}) = - \sin (- \frac{\pi}{3})$, so that we need the modified version of the 
handle 
to do the Lagrangian surgery. 
We will use the one defined by the smoothing of:
$$\left \{ { \left. e^{-t} \left( { \begin{array}{c} x_0\\x_1 \end{array} } \right) 
+ e^{t} \left( { \begin{array}{c} e^{i \frac{\pi}{3}} x_0\\ 
e^{-i \frac{\pi}{3}} (-x_1) \end{array} } \right) 
\right|   \begin{array}{c} t \in \rr \\x_0^2+x_1^2 = \varepsilon_1^2 \end{array} } \right \}.$$

At $[0:1:0]$, the local chart is 
$$[z_0:z_1:z_2] \mapsto \left( \frac{z_0}{z_1}, \frac{z_2}{z_1} \right),$$ so that 
locally, $L_0$ is the real plane 
$$l_0 = \{(x_0, x_2) | \, x_0, x_1 \in \rr\}$$ 
and $L_1$ is
$$l_1 = \{(e^{i \frac{2\pi}{3}} x_0, e^{i \frac{\pi}{3}} x_2) | \, x_0, x_2 \in \rr\}.$$
As $\sin (2\pi/3) = \sin (\pi/3)$, we can use the first description of the handle to define 
the Lagrangian surgery:
$$\left \{ { \left. e^{-t} \left( { \begin{array}{c} x_0\\x_2 \end{array} } \right) 
+ e^{t} \left( { \begin{array}{c} e^{i \frac{2\pi}{3}} x_0\\ e^{i \frac{\pi}{3}} x_2 \end{array} } \right) 
\right|   \begin{array}{c} t \in \rr \\x_0^2+x_2^2 = \varepsilon_1^2 \end{array} } \right \}.$$

At $[1:0:0]$, the local chart is 
$$[z_0:z_1:z_2] \mapsto \left( \frac{z_1}{z_0}, \frac{z_2}{z_0} \right),$$ 
so that locally, $L_0$ is the real plane 
$$l_0 = \{(x_1,x_2) | \, x_1,x_2 \in \rr\}$$
 and $L_1$ is 
$$l_1 = \{(e^{- i \frac{2\pi}{3}}x_1, e^{- i \frac{\pi}{3}} x_2) | \, x_1, x_2 \in \rr\}.$$
As $\sin (-2\pi/3) = \sin (-\pi/3)$, we can also use the first description of the handle 
to do the Lagrangian surgery. But for the monotonicity condition to be satisfied in 
Section~\ref{sec:KSigma-monotonicity}, we will use instead the smoothing of the~$h'$-handle:
$$\left \{ { \left. e^{-t} \left( { \begin{array}{c} x_1 \\ x_2 \end{array} } \right) 
+ e^{t} \left( { \begin{array}{c} e^{- i \frac{2\pi}{3}} (-x_1) \\ e^{-i \frac{\pi}{3}} (-x_2) \end{array} } \right) 
\right|   \begin{array}{c} t \in \rr \\x_1^2+x_2^2 = \varepsilon_1^2 \end{array} } \right \}.$$

After these surgeries, the projection of the Lagrangian~$L$ we constructed will be 
contained in the polytope obtained from the polytope of~$\plc$ by cutting the three vertices 
as in Figure~\ref{fig:PolytopeKSigma2}.

\begin{figure}[htbp]
  \begin{center}
   \psfrag{3}{$3$}
   \psfrag{eps1}{$\varepsilon_1^2$} 
   \psfrag{3-eps1}{$3-\varepsilon_1^2$} 
   \includegraphics[height=6cm]{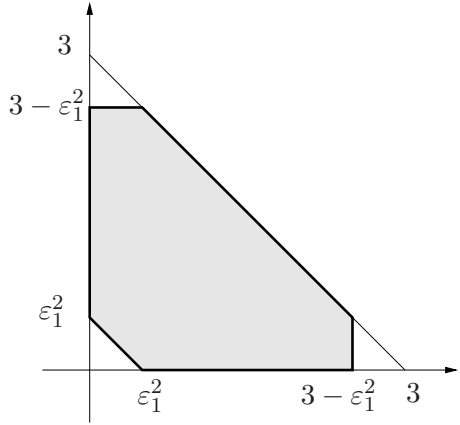}
   \caption{The image of~$L$ under the moment map is contained in and smoothly approximates the shaded polytope.} 
   \label{fig:PolytopeKSigma2}
     \end{center}
\end{figure}

The choice of these copies of~$\plr$ and these surgeries is motivated by the monotonicity 
condition we aim to prove for this construction in the next section. 
Let us describe the restriction 
of the surgery along the coordinate-$\drc$s, that is the projective lines which are the 
preimages of the edges on the boundary of the moment polytope and can be defined by the 
vanishing of one of the homogeneous coordinates. We will describe the case of 
$$\{[z_0 : z_1 : z_2] \mid z_0 = 0 \},$$
the other coordinate-$\drc$s being similar.

Along the sphere~$z_0 = 0$, $L_0$ and~$L_1$ are two circles intersecting transversally 
at the north and the south pole. 
Locally in the chart~$\cc$ at~$[0:0:1]$, we have~$L_0$ on the real axis and~$L_1$ on 
the axis~$e^{-i \frac{\pi}{3}} \rr$. 
In this chart, the intersection of the chosen surgery with~$\{z_0 = 0\}$ consists in two curves 
(see Figure~\ref{fig:En001}-left): inside a ball centered at the origin and of 
area~$\varepsilon$ 
they lie in two opposite "quadrants" defined by these two axes, 
that is in the quadrants making an angle of~$2 \frac{\pi}{3}$.

\begin{figure}[htbp]
  \begin{center}
   \psfrag{L_0}{$\scriptstyle{L_0}$}
   \psfrag{L_1}{$\scriptstyle{L_1}$}
   \psfrag{L}{$\scriptstyle{L}$}   
   \includegraphics[height=4cm]{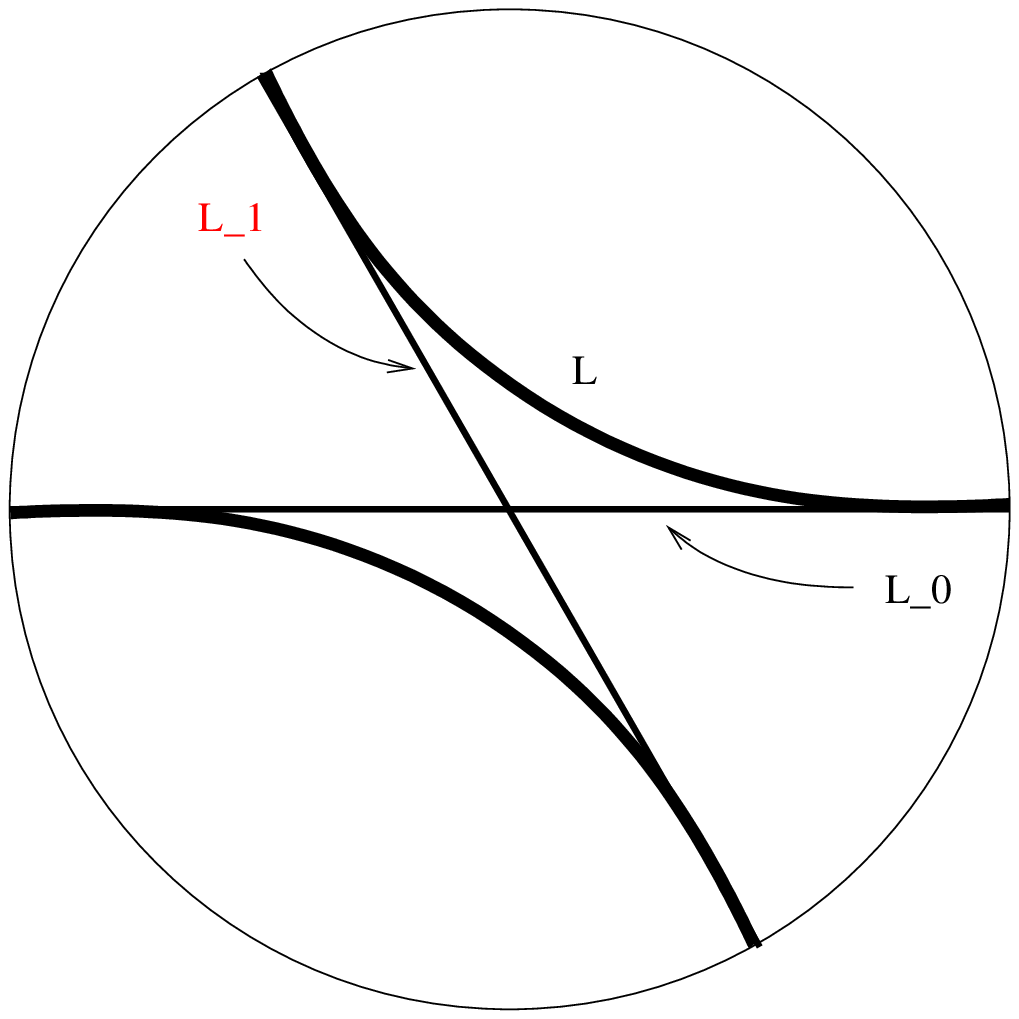} \hspace{2.5cm} \includegraphics[height=4cm]{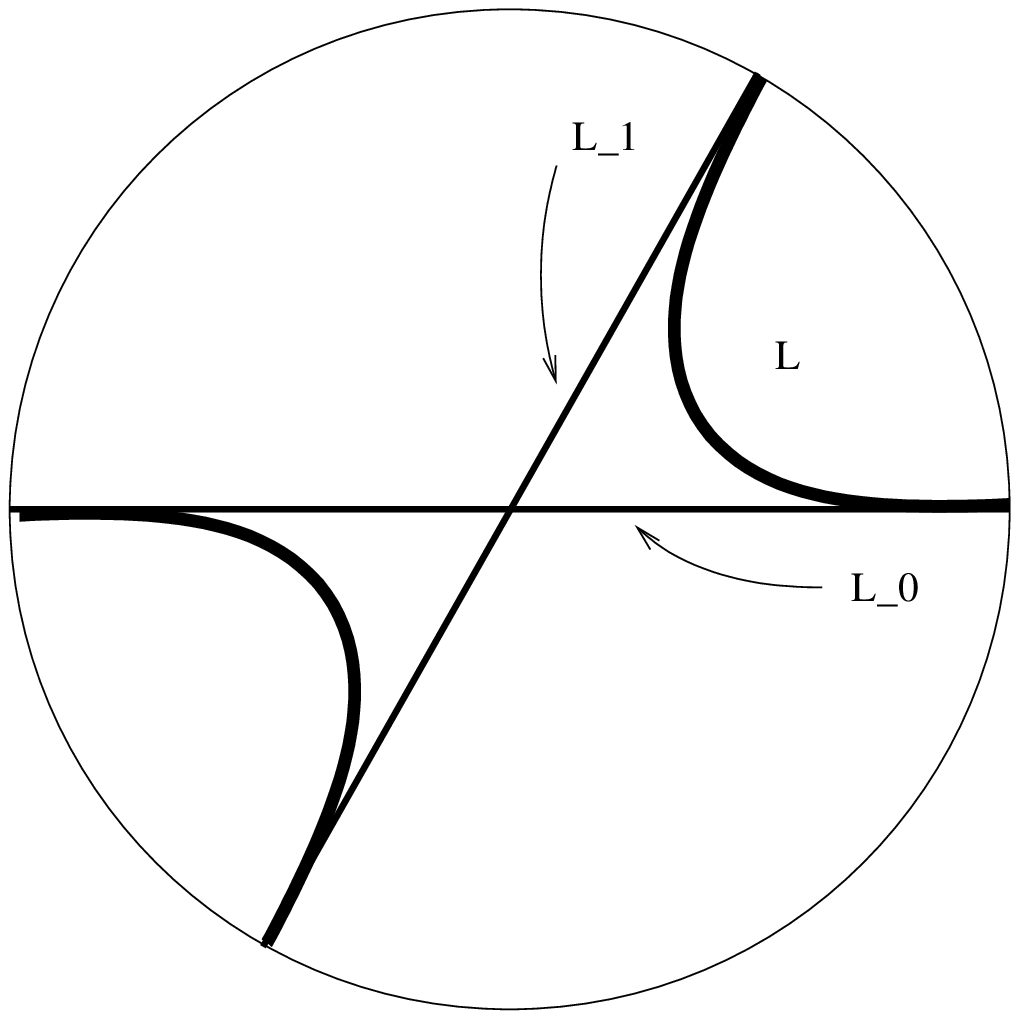} 
   \caption{The surgery along~$z_0 = 0$ in the chart at~$[0:0:1]$ (on the left) 
   and in the chart at~$[0:1:0]$ (right)} 
   \label{fig:En001}
     \end{center}
\end{figure}

Locally at~$[0:1:0]$ we have a similar picture but with curves in the quadrants 
making an angle of~$\frac{\pi}{3}$ (see Figure~\ref{fig:En001}-right).

Away from the small neighbourhoods where we do the surgery, namely on the part where 
we glue 
the two charts, the restriction of~$L$ is the restriction of~$L_0$ and~$L_1$ to this 
complex projective line. 
One sees then that the restriction of~$L$ to this~$\drc$ is one circle joining~$L_0$ 
and~$L_1$ through the two handles constructed at each intersection point and looking 
like the seam of a tennis ball (see Figure~\ref{fig:Tennis}). 

\begin{figure}[htbp]
  \begin{center}
   \psfrag{L_0}{$\scriptstyle{L_0}$}
   \psfrag{L_1}{$\scriptstyle{L_1}$}
   \includegraphics[height=6cm]{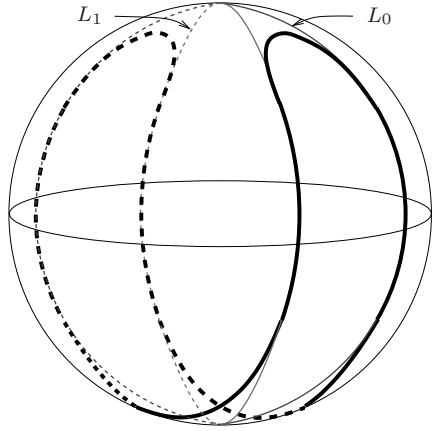}
   \caption{The intersection of the surgery with~$\{z_0 = 0\}$} 
   \label{fig:Tennis}
     \end{center}
\end{figure}

If we had chosen the first handle we described 
in Section~\ref{sec:surgeryatapoint} at~$[0:0:1]$, the restriction of~$L$ to~$z_0 = 0$ 
would have been the union of two circles. 

Actually, the choice of handles we made is such that in each of the other 
coordinate-$\drc$s (namely~$z_1 = 0$ and~$z_2 = 0$), the restriction of~$L$ 
is one circle joining~$L_0$ and~$L_1$ through the two handles constructed at each 
intersection point at the poles of~$\drc$. 
Indeed, in the case of two circles in the intersection with a coordinate-$\drc$, one cannot 
expect to satisfy the monotonicity condition. 
Note however that as the two intersecting Lagrangians~$L_0$ and~$L_1$ are not orientable, 
it follows from~\cite[Proposition 2]{MR1097259} that the topology of the Lagrangian 
we obtain after surgeries does not depend on the choice of handles.

\subsubsection{Monotonicity}
\label{sec:KSigma-monotonicity}

Let us normalize the symplectic form on~$\plc$ such that the area of a projective 
line is~$3$:
$$\int_{\drc} \omega = 3.$$

With this normalization, $\plc$ is monotone with monotonicity constant~$1$:
$$\forall v \in H_2(\plc), \int_v \omega = c_1(T\plc)(v)$$
and any monotone Lagrangian submanifold~$L$ will have monotonicity constant equal 
to half the monotonicity constant of~$\plc$ (see~\cite{Oh_I}), namely~$\frac{1}{2}$:
$$\forall u \in H_2(\plc, L), \int_u \omega = \frac{1}{2} \, \mu_L(u),$$
where~$\mu_L$ is the Maslov class of~$L$.

\begin{thm}
The construction of Section~\ref{sec:KSigma-construction} produces a monotone Lagrangian 
embedding of the connected sum of a Klein bottle and 
a compact orientable surface of genus~$2$ in~$\plc$.
\end{thm}

\begin{proof}
Topologically, the surgery at a point between two copies of the real projective plane gives 
the connected sum of these two spaces, namely a Klein bottle. Then attaching a 
$2$-dimensional handle corresponds to a connected sum with a torus, so that the 
Lagrangian submanifold constructed in~\ref{sec:KSigma-construction} is diffeomorphic 
to the connected sum 
$$L \cong \plr \# \plr \# \TT^2 \# \TT^2 \cong \ K \# \Sigma_2$$ 
where~$K$ is a Klein bottle and~$\Sigma_2$
is a compact orientable surface of genus~$2$.

We know that~$H_2(\plc, L) \cong H_2(\plc) \oplus H_1(L)$ as~$H_2(L) = 0$.  
Let us examine the monotonicity condition on each factor of this direct sum.

On the factor~$H_2(\plc)$, we already have the monotonicity condition from the one 
on~$\plc$ so that we only need to verify the monotonicity condition on the disks 
representing generators of~$H_1(L)$.
This means also that for a given generator of~$H_1(L)$, it is enough to satisfy 
the monotonicity 
condition for one choice of disk with boundary this generator, as the 
symplectic invariants for another disc with the same boundary will differ by the 
invariants coming from a sphere (the sphere obtained by gluing the two disks 
along their boundary) where the condition is already verified.

Now~$H_1(L) = \zz^5 \oplus \zz/2$ can be generated by the two loops generating the
first homology group of each copy of~$\plr$, the three loops inside each handle 
generating the homology of the handle and three loops "between" the handles 
(see Figure~\ref{fig:3loops}). 

\begin{figure}[htbp]
  \begin{center}
   \psfrag{RP2}{$\plr$}
   \includegraphics[height=8cm]{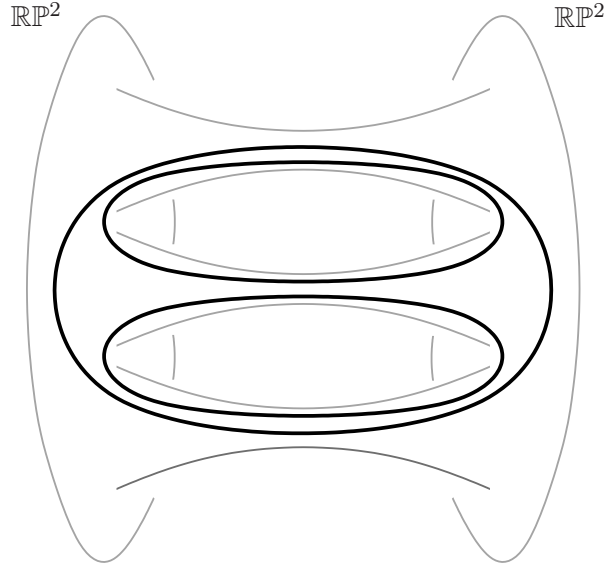}
   \caption{Three loops between the handles} 
   \label{fig:3loops}
     \end{center}
\end{figure}

Any loop siting in one of the original copies of~$\plr$, $L_0$ or~$L_1$, 
satisfies the monotonicity condition because~$L_0$ and~$L_1$ are monotone.

As a representative of a generator of the homology of the handle, one can take 
the circle on one of the extremities of the handle lying on one of the copies 
of the Lagrangian~$\plr$, say for example~$L_0$.
More explicitely, one may take the image of the circle 
$\{(x_0, x_1), \, x_0^2 + x_1^2 = \varepsilon^2 \}$, for~$\varepsilon > \varepsilon_1$ 
small, in the chart around an intersection point of~$L_0$ and~$L_1$.
We take a disc in~$\plc$ with boundary this circle in~$L$ and compute the symplectic 
invariants of that disc. One can for example choose the disc in~$L_0$ which was cut out 
from~$L_0$ to built~$L$. But as the disc is Lagrangian, the two invariants, area and 
Maslov class, vanish on this disc.

One is left with checking the monotonicity condition on the circles between handles. 
One can prove that the monotonicity condition is satisfied on the three circles 
drawing the tennis ball seam on the coordinate-$\drc$s we described at the end of the previous 
section and the discs they bound on these~$\drc$. 
Unfortunately these circles are not in our set of generators for~$H_1(L)$, since their homology class is $2$ times the loop between the corresponding handles depicted
in Figure~\ref{fig:3loops}.
But we can use for the generators loops which partially follow these seams. 
Let us describe a loop~$\gamma$ we can choose between the handles created at~$[0:0:1]$ 
and~$[0:1:0]$ and a disk it bounds. 
Two other loops between handles can be constructed in a similar way.

The loop~$\gamma$ is almost entirely lying in the coordinate~$\drc$ of homogeneous 
equation $z_0 = 0$. See Figure~\ref{fig:Loopgamma}.

\begin{figure}[htbp]
  \begin{center}
   \psfrag{L_0}{$L_0$}
   \psfrag{L_1}{$L_1$}
   \psfrag{L}{$L$}  
   \psfrag{a}{$a$}   
   \psfrag{b}{$b$}   
   \psfrag{c}{$c$}   
   \psfrag{d}{$d$}   
   \psfrag{e}{$e$}    
   \includegraphics[height=6cm]{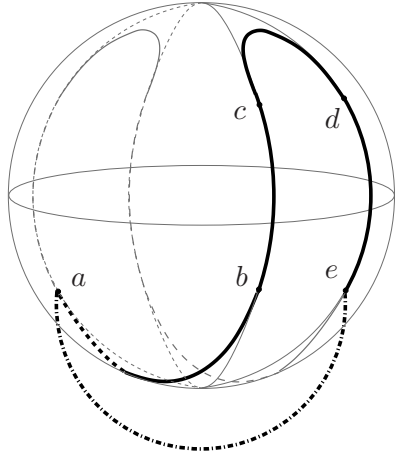} 
   \caption{The loop~$\gamma$ 
   } 
   \label{fig:Loopgamma}
     \end{center}
\end{figure}

It is based at a point where~$L$ coincides with~$L_0$, for example the point~$a$ of local 
coordinates~$(z_0, z_1) = (0, \varepsilon)$ in the local chart at~$[0:0:1]$. 
From this point, follow the handle at~$[0:0:1]$ along the path parametrized by 
$$\{(0, e^{-t} \varepsilon_1 + e^{t} e^{-i \frac{\pi}{3}} (-\varepsilon_1))\}$$
(we include here the smoothing by considering we can locally take the 
parametrization of the handle for~$t$ varying from~$-\infty$ to~$+\infty$) 
till the point~$b$ of local coordinates~$(z_0, z_2) = (0, e^{-i \frac{\pi}{3}} \varepsilon)$.
Then follow~$L_1 \cap \{z_0 = 0\}$ "up" towards~$[0:1:0]$ till the 
point~$c$ of local coordinates~$(z_0, z_2) = (0, e^{i \frac{\pi}{3}} (-\varepsilon))$ in the 
local chart at~$[0:1:0]$. 
Next, $\gamma$ goes back to~$L_0$ through the handle at~$[0:1:0]$, following 
"backwards" the path parametrized by 
$$(0, e^{-t} (-\varepsilon_1) + e^{t} e^{i \frac{\pi}{3}} (-\varepsilon_1))$$
in local coordinates in the chart at~$[0:1:0]$ till it reaches the 
point~$d$ of local coordinates~$(0, - \varepsilon)$ in the chart at~$[0:0:1]$.
The path then follows $L_0 \cap \{z_0 = 0\}$ "down" to~$[0:0:1]$ till the point~$e$
of local coordinates~$(0, - \varepsilon)$ in the chart at~$[0:0:1]$.
Now we close the loop~$\gamma$ with a path contained in~$L_0 \cap L$ but leaving the 
coordinate-$\drc$ $\{z_0 = 0\}$ by following the half circle parametrized in the 
chart at~$[0:0:1]$ by
$$\{(-\varepsilon \sin (t), -\varepsilon \cos (t)) | \, t \in [0;\pi] \}.$$

This loop encloses a disk~$u$ in $\plc$ which can be described as the union of the 
portion of sphere~$\{z_0 = 0\}$ lying between~$L_0$ and~$L_1$ in the sector making 
a $\frac{\pi}{3}$-angle and delimited by~$\gamma$ in the "north", the portion of the same sphere in 
the $\frac{2\pi}{3}$-sector between~$\gamma$ and the segment in~$L_0$ (but not~$L$) 
of coordinates in the chart at~$[0:0:1]$
$$[e;a] = \{(0,z_2) | \, z_2 \in [-\varepsilon, \varepsilon]\},$$
and the half disk in~$L_0$ enclosed by this segment~$[e;a]$ and~$\gamma$ 
(the part of the disk~$u$ in~$\{z_0 = 0\}$ and the half disk are glued along the 
segment~$[e;a]$).

This disk~$u$ and the similar ones we can build between the other handles together 
with the disks considered before generate~$H_2(\plc,L)$, so that the monotonicity of~$L$ 
will follow from the next two lemmas which compute the area and the 
Maslov class of~$u$. 
\end{proof}

\begin{lem}
\label{lem:u-area}
The disk~$u$ has area~$\dfrac{1}{2}$.
\end{lem}

\begin{proof}
We will compute the area of~$u$ by adding the area of the  
three portions we described above. 

As we noticed in Section~\ref{sec:surgeryatapoint}, we can make the Lagrangian surgery 
such that the areas between the restrictions to each~$\cc$-factor of~$\cc^2$ of the 
original Lagrangians and the handles 
 are small and equal. We do the surgeries as in Section~\ref{sec:surgeryatapoint} 
so that these areas are equal 
to~$a(\varepsilon)$ small at each of the intersection points. 

The area of the first portion is then the area of the~$\frac{\pi}{3}$-sector, namely one sixth 
of the total area of the sphere, minus the area lost at the handle, namely~$a(\varepsilon)$.
The area in the other sector of the coordinate projective line is the area gained through 
the handle, that is~$a(\varepsilon)$. The contribution of the portions of the disk 
in the coordinate sphere is thus~$\frac{1}{2}$.
The contribution of the half-disk in~$L_0$ is zero as this half-disk lies totally in a 
Lagrangian submanifold.
\end{proof}

\begin{lem}
\label{lem:u-Maslov}
The disk~$u$ has Maslov class~$1$.
\end{lem}

\begin{proof}
The first crucial remark is that the disk~$u$ is lying entirely in the chart 
at~$[0:0:1]$, so that the tangent bundle of~$\plc$ is already trivialized along the disk 
when we work in this chart. 
Now, to compute the Maslov class of this disk, we will 
write the loop in the Lagrangian Grassmannian we have along the boundary~$\gamma$ as the 
action of a loop~$A(t)$ of matrices in~$U(2)$ on the reference linear Lagrangian
space~$\rr^2$ of~$\cc^2$. 
The Maslov class~$\mu(u)$ is then the degree of the square of 
the determinant of~$A$ seen as a map from~$\cerc$ to~$\cerc$.

To describe this action, we will decompose the action along the different portions 
of the loop we considered above. The loop~$\gamma$ is the concatenation of the 
paths~$\gamma_1$ from~$a$ to $b$, $\gamma_2$ from~$b$ to~$c$, ... and, $\gamma_5$ 
from~$e$ to~$a$. On each of these paths, we will decompose the action of~$U(2)$ 
so that the Maslov class of~$u$ can be written as the product of these different paths 
of matrices.

At the point~$a$, we are in $L_0 \cap L$, with~$L_0$ 
a linear Lagrangian in the chart, so that the submanifold identifies with its tangent 
space, namely~$\rr^2$, our reference linear Lagrangian subspace.

Between the points~$b$ and~$c$, we stay on~$L_1$, the tangent space is identically 
equal to the linear Lagrangian 
subspace~$l_1= \{(e^{i \frac{\pi}{3}} x_0, e^{- i \frac{\pi}{3}} x_1) | \, x_0, x_1 \in \rr\}$ 
so that~$A(t)$ is the identity along~$\gamma_2$ and this portion has no contribution 
to the degree.

Similarly, along~$\gamma_4$ and~$\gamma_5$, we stay on the same linear 
Lagrangian (either~$l_1$ or~$l_0$) so that the matrix~$A(t)$ is again the identity 
along these portions of~$\gamma$.

We are left to compute the contributions of the handles to the Maslov class.

Computing the contribution of the handle at~$[0:0:1]$ along~$\gamma_1$ 
is straightforward because we 
have the parametrization of the handle explicitely written in this chart, namely
$$\left \{ { \left. e^{-t} \left( { \begin{array}{c} x_0 \\x_1 \end{array} } \right) 
+ e^{t} \left( { \begin{array}{c} e^{i \frac{\pi}{3}} x_0 \\ e^{-i \frac{\pi}{3}} (-x_1) \end{array} } \right) 
\right|   \begin{array}{c} t \in \rr \\x_0^2+x_1^2 = \varepsilon_1^2 \end{array} } \right \}.$$
Then the tangent spaces to that handle along the points in~$\{z_0 = 0\}$ can be parametrized 
by
$$\left \{ { \left. e^{-t} \left( { \begin{array}{c} X_0 \\- \varepsilon_1 T \end{array} } \right) 
+ e^{t} \left( { \begin{array}{c} e^{i \frac{\pi}{3}} X_0 \\ e^{-i \frac{\pi}{3}} (- \varepsilon_1) T \end{array} } \right) 
\right|  }  t, T, X_0 \in \rr \right \}.$$
For~$t$ going to~$- \infty$, the tangent space is asymptotic to~$\rr^2$ for which we can 
take the canonical basis~$\{ (1,0) , (0,1) \}$. Through the handle, the 
vectors~$(X_0,-\varepsilon_1 T) = (1,0)$ and $(X_0,-\varepsilon_1 T) = (0,1)$ 
are mapped to $(X_0,-\varepsilon_1 T) = (e^{i \frac{\pi}{3}},0)$ and 
$(X_0,-\varepsilon_1 T) = (0,e^{-i \frac{\pi}{3}})$, a basis of~$l_1$ through a path of 
matrices homotopic to
$$A_1(s) = \left( { \begin{array}{cc} e^{-i s \frac{\pi}{3}} & \\0 & e^{i s \frac{\pi}{3}} \end{array} } \right)$$
for~$s$ going from~$s = 0$ to $s = 1$. The determinant of~$A_1$ being identically equal to~$1$, 
this part of~$\gamma$ will not contribute to the degree.

The contribution of the handle at~$[0:1:0]$ can also be computed thanks to the 
explicit parametrization of the handle, but we have first to write it in the chart at~$[0:0:1]$ 
for our computation.
In that chart, the handle is now parametrized by
$$\left \{ { \left. 
\left({ \dfrac{ e^{-t} x_0 + e^{t} e^{i \frac{2 \pi}{3}} x_0 }{ e^{-t} x_2 + e^{t} e^{i \frac{\pi}{3}} x_2 }, 
 \dfrac{ 1 }{ e^{-t} x_2 + e^{t} e^{i \frac{\pi}{3}} x_2 }} \right) 
 \right|   \begin{array}{c} t \in \rr \\x_0^2+x_2^2 = \varepsilon_1^2 \end{array} } \right \},$$
so that the tangent spaces are described by 
$$\left \{ { \left. {
\left({ \dfrac{e^{-t} X_0 + e^{t} e^{i \frac{2\pi}{3}} X_0}
{e^{-t} x_2 + e^{t} e^{i \frac{\pi}{3}} x_2}, 
- \dfrac{-e^{-t} x_2 T + e^{t} e^{i \frac{\pi}{3}} x_2 T }{(e^{-t} x_2 + e^{t} e^{i \frac{\pi}{3}} x_2)^2}} \right) 
} \right| \begin{array}{c} 
t, T, X_0 \in \rr \\ x_2 = -\varepsilon_1 \end{array} 
} \right \}.$$
When~$t$ tends to~$-\infty$, the handle is indeed asymptotic to~$l_0$ and 
when~$t$ tends to~$+\infty$, the handle is asymptotic 
to~$l_1 = \{(e^{i \frac{\pi}{3}} x_0, e^{- i \frac{\pi}{3}} x_1) | \, x_0, x_1 \in \rr\}$. 
The canonical basis~$\{ (1,0) , (0,1) \}$ is mapped through the handle 
to$\{ (e^{i \frac{\pi}{3}},0) , (0,- e^{-i \frac{\pi}{3}})\}$.
The action on the first coordinate can be homotopic to
the path in~$U(1)$ $s \in [0,1] \mapsto e^{i s \frac{\pi}{3}}$ or 
to~$s \in [0,1] \mapsto e^{-i s \frac{5\pi}{3}}$.
But for~$t = 0$ at the middle of the handle, one can check via the formula 
that the image of the vector~$(1,0)$ is positively proportional to 
the vector~$(e^{i \frac{\pi}{6}},0)$ so that the path is 
$$s \in [0,1] \mapsto e^{i s \frac{\pi}{3}}.$$
For the second coordinate, in a similar manner one can act either by a 
path homotopic  to~$s \in [0,1] \mapsto e^{i s \frac{ 2 \pi}{3}}$ or 
to~$s \in [0,1] \mapsto e^{- i s \frac{4\pi}{3}}$.
For~$t = 0$ at the middle of the handle, one can check via the formula that 
the image of the vector~$(0,1)$ is positively proportional to 
the vector~$(0,-e^{i \frac{\pi}{3}})$ so that the path is 
$$s \in [0,1] \mapsto e^{- i s \frac{4\pi}{3}}.$$
The contribution of the handle along~$\{z_0 = 0\}$ from~$l_0$ to~$l_1$ is 
thus homotopic to
$$s \in [0,1] \mapsto 
\left( { \begin{array}{cc} e^{i s \frac{\pi}{3}} & \\0 & e^{-i s \frac{4\pi}{3}} \end{array} } \right).$$
But along~$\gamma_3$ we move from~$L_1$ to~$L_0$ so that the matrix of the action of~$U(2)$ 
along this portion of the boundary is homotopic to
$$s \in [0,1] \mapsto 
\left( { \begin{array}{cc} e^{-i s \frac{\pi}{3}} & \\0 & e^{+4i s \frac{\pi}{3}} \end{array} } \right),$$
whose determinant squared is equal to 
$$\begin{array}{ccc}
 [0,1]& \longrightarrow & \cerc  \\
    s & \longmapsto     & e^{2i s \pi}
\end{array}$$
which is of degree~$1$.
In conclusion, the Maslov class of~$u$ being the sum of all the contributions of the portions 
of the loop~$\gamma$ is equal to~$1$.
\end{proof}

\subsection{A monotone~$K \# \Sigma_4$ in~$\drc \times \drc$}

\label{sec:KSigma4}

In this section, we explain the construction of a monotone Lagrangian embedding 
of the connected sum of a Klein bottle and a surface of genus~$4$ in the 
product~$\drc \times \drc$.\\

Let us normalize the symplectic form on~$\drc \times \drc$ such that the area of 
a projective line is~$2$:
$$\int_{\drc} \omega = 2.$$
With this normalisation, $\drc \times \drc$ is monotone with monotonicity constant~$1$:
$$\forall v \in H_2(\drc \times \drc), \int_v \omega = c_1(T(\drc \times \drc))(v),$$
and any monotone Lagrangian submanifold~$L$ will have a monotonicity constant 
equal to~$\frac{1}{2}$:
$$\forall u \in H_2(\drc \times \drc, L), \int_u \omega = \frac{1}{2} \, \mu_L(u).$$
To construct a monotone~$K \# \Sigma_4$ in~$\drc \times \drc$, we take two Hamiltonian 
isotopic copies of the real part of~$\drc \times \drc$, i.e. two Lagrangian tori, 
that intersect in four points and perform a suitable Lagrangian surgery at 
these four points.

The image of the resulting Lagrangian~$L$ under the moment map will be contained in the original 
moment polytope of the ambient symplectic manifold choped at its four corners (see 
Figure~\ref{fig:PolytopeKSigma4}).

\begin{figure}[htbp]
  \begin{center}
   \psfrag{2}{$2$}
   \psfrag{eps1}{$\varepsilon_1^2$}  
   \psfrag{2-eps1}{$2-\varepsilon_1^2$} 
   \includegraphics[height=6cm]{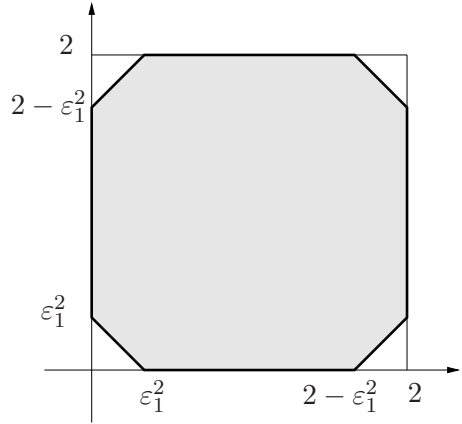}
   \caption{The image of~$L$ under the moment map is contained in and smoothly approximates the shaded polytope.} 
   \label{fig:PolytopeKSigma4}
     \end{center}
\end{figure}

In view of the monotonicity condition, we have here two possible choices for the copies 
of the real part and corresponding surgeries.

One choice is to take one copy of the real part to be
$$L_0 = \{ \left. ([x_0:x_1], [u_0:u_1]) \in \drc \times \drc \right|  
x_0, x_1, u_0, u_1 \in \rr \}$$ 
and the second one its "rotation" by~$i$:
$$L_1 = \{ \left. ([e^{i \frac{\pi}{2}} x_0 : x_1], [e^{i \frac{\pi}{2}} u_0:u_1]) \in \drc \times \drc \right|  x_0, x_1, u_0, u_1 \in \rr \}.$$
They intersect each other in the four points 
$$([0:1], [0:1]), ([1:0], [0:1]), ([0:1], [1:0]), ([1:0], [1:0]),$$
the preimage of the four corners under the standard moment map.
Then choose the following handles at the intersection points:
\begin{itemize}
 \item at~$([0:1], [0:1])$, $l_0 = \{(x_0,u_0) | \, x_0, u_0 \in \rr\}$, $l_1 = \{(i x_0, i u_0) | \, x_0, u_0 \in \rr\}$, and we insert the handle:
$$\left \{ { \left. e^{-t} \left( { \begin{array}{c} x_0 \\u_0 \end{array} } \right) + e^{t} \left( { \begin{array}{c} i x_0 \\ i u_0 \end{array} } \right) \right|  
 \begin{array}{c} t \in \rr \\x_0^2+u_0^2 = \varepsilon_1 \end{array} } \right \},$$
 \item at~$([1:0], [0:1])$, $l_0 = \{(x_1,u_0) | \, x_1, u_0 \in \rr\}$, $l_1 = \{(- i x_1, i u_0) | \, x_1, u_0 \in \rr\}$, and we insert the handle:
$$\left \{ { \left. e^{-t} \left( { \begin{array}{c} x_1 \\u_0 \end{array} } \right) + e^{t} \left( { \begin{array}{c} - i (- x_1) \\ i u_0 \end{array} } \right) \right|   
\begin{array}{c} t \in \rr \\x_1^2+u_0^2 = \varepsilon_1 \end{array} } \right \},$$
 \item at~$([0:1], [1:0])$, $l_0 = \{(x_0,u_1) | \, x_0, u_1 \in \rr\}$, $l_1 = \{(i x_0, - i u_1) | \, x_0, u_1 \in \rr\}$, and we insert the handle:
$$\left \{ { \left. e^{-t} \left( { \begin{array}{c} x_0 \\u_1 \end{array} } \right) + e^{t} \left( { \begin{array}{c} i x_0 \\ - i (- u_1) \end{array} } \right) \right|   
\begin{array}{c} t \in \rr \\x_0^2+u_1^2 = \varepsilon_1 \end{array} } \right \},$$
 \item at~$([1:0], [1:0])$, $l_0 = \{(x_1,u_1) | \, x_1, u_1 \in \rr\}$, $l_1 = \{(- i x_1, - i u_1) | \, x_1, u_1 \in \rr\}$, and we insert the handle:
$$\left \{ { \left. e^{-t} \left( { \begin{array}{c} x_1 \\u_1 \end{array} } \right) + e^{t} \left( { \begin{array}{c} - i (- x_1) \\ - i (- u_1) \end{array} } \right) \right|  
 \begin{array}{c} t \in \rr \\x_1^2+u_1^2 = \varepsilon_1 \end{array} } \right \},$$
\end{itemize}
so that the intersection of the Lagrangian obtained by surgery with any of the $\drc$ 
which are preimages of the edges on the boundary of the moment polytope 
is one circle of the shape of the tennis ball seam as before. 

As in the previous section, one cannot take this circle 
to check the monotonicity of the surface as it goes around the handle twice, 
but we can construct similar loops going only once around the handle and 
check that they satisfy the monotonicity condition.\\

Alternatively, one can choose the surgeries such that the intersection of the 
Lagrangian after surgery with each coordinate-$\drc$ is a union of two circles. 
To ensure 
monotonicity in this case, we need each circle to enclose an area slightly 
bigger than the one we get with the choice of~$L_1$ above.

We will take the same Lagrangian~$L_0$ and for~$L_1$ the following Hamiltonian isotopic copy:
$$L_1 = \left \{ \left. \left([e^{i \left(\frac{\pi}{2}+\delta \right)} x_0:x_1], [e^{i \left(\frac{\pi}{2}+\delta \right)} u_0:u_1] \right) \in \drc \times \drc \right|  x_0, x_1, u_0, u_1 \in \rr \right \},$$
for a small positive parameter~$\delta$ that will be fixed later.
They intersect each another again in the four corners of the moment polytope 
of~$\drc \times \drc$.

In each chart around the intersection points, we are making the following choices:
\begin{itemize}
 \item at~$([0:1], [0:1])$, $L_0$ and $L_1$ are the linear subspaces $l_0 = \{(x_0,u_0) | \, x_0, u_0 \in \rr\}$, 
 $l_1 = \{(e^{i \left(\frac{\pi}{2}+\delta \right)} x_0, e^{i \left(\frac{\pi}{2}+\delta \right)} u_0) | \, x_0, u_0 \in \rr\}$, 
 and we insert the handle:
$$\left \{ { \left. e^{-t} \left( { \begin{array}{c} x_0 \\u_0 \end{array} } \right) 
+ e^{t} \left( { \begin{array}{c} e^{i \left(\frac{\pi}{2}+\delta \right)} x_0 \\ e^{i \left(\frac{\pi}{2}+\delta \right)} u_0 \end{array} } \right) \right|  
 \begin{array}{c} t \in \rr \\x_0^2+u_0^2 = \varepsilon_1 \end{array} } \right \};$$
 \item at~$([1:0], [0:1])$, $L_0$ and $L_1$ are the linear subspaces  $l_0 = \{(x_1,u_0) | \, x_1, u_0 \in \rr\}$, 
 $l_1 = \{(e^{-i \left(\frac{\pi}{2}+\delta \right)} x_1, e^{i \left(\frac{\pi}{2}+\delta \right)} u_0) | \, x_1, u_0 \in \rr\}$, and we insert the handle:
$$\left \{ { \left. e^{-t} \left( { \begin{array}{c} x_1 \\u_0 \end{array} } \right) + e^{t} \left( { \begin{array}{c} e^{-i \left(\frac{\pi}{2}+\delta \right)} x_1 \\ e^{i \left(\frac{\pi}{2}+\delta \right)} (-u_0) \end{array} } \right) \right|   
\begin{array}{c} t \in \rr \\x_1^2+u_0^2 = \varepsilon_1 \end{array} } \right \};$$
 \item at~$([0:1], [1:0])$, $L_0$ and $L_1$ are the linear subspaces  $l_0 = \{(x_0,u_1) | \, x_0, u_1 \in \rr\}$, 
 $l_1 = \{(e^{i \left(\frac{\pi}{2}+\delta \right)} x_0, e^{-i \left(\frac{\pi}{2}+\delta \right)} u_1) | \, x_0, u_1 \in \rr\}$, 
 and we insert the handle:
$$\left \{ { \left. e^{-t} \left( { \begin{array}{c} x_0 \\u_1 \end{array} } \right) 
+ e^{t} \left( { \begin{array}{c} e^{i \left(\frac{\pi}{2}+\delta \right)} (-x_0) \\ e^{-i \left(\frac{\pi}{2}+\delta \right)} u_1 \end{array} } \right) \right|   
\begin{array}{c} t \in \rr \\x_0^2+u_1^2 = \varepsilon_1 \end{array} } \right \};$$
 \item at~$([1:0], [1:0])$, $L_0$ and $L_1$ are the linear subspaces  $l_0 = \{(x_1,u_1) | \, x_1, u_1 \in \rr\}$, 
 $l_1 = \{(e^{-i \left(\frac{\pi}{2}+\delta \right)} x_1, e^{-i \left(\frac{\pi}{2}+\delta \right)} u_1) | \, x_1, u_1 \in \rr\}$, 
 and we insert the handle:
$$\left \{ { \left. e^{-t} \left( { \begin{array}{c} x_1 \\u_1 \end{array} } \right) 
+ e^{t} \left( { \begin{array}{c} e^{-i \left(\frac{\pi}{2}+\delta \right)} (- x_1) \\ 
e^{-i \left(\frac{\pi}{2}+\delta \right)} (- u_1) \end{array} } \right) \right|  
 \begin{array}{c} t \in \rr \\x_1^2+u_1^2 = \varepsilon_1 \end{array} } \right \}.$$
\end{itemize}

\begin{thm}
The construction produces a monotone Lagrangian embedding of a compact 
surface $K \# \Sigma_4$ in~$\drc \times \drc$ for an appropriate choice of~$\delta$.
\end{thm}

\begin{proof}
Note that even though the two Lagrangian submanifolds~$L_0$ and~$L_1$ are oriented 
(they are tori), one can check that given an orientation of the two tori, two of these handles 
do not preserve the orientation (this cannot be avoided, it is related to the fact 
that the signs of the intersection points cancel in pairs for any choice of orientation). 
Therefore, we get through these four surgeries a non-orientable Lagrangian which is 
the connected sum of the two initial tori with one torus and two Klein bottles. 
It is diffeomorphic to
$$L \cong K \#  \Sigma_4.$$

Following the remarks from Section~\ref{sec:surgeryatapoint}, we can do the surgery 
in each corner of the moment map such that the areas between the handle and the initial 
Lagrangians are small and equal along each~$\drc$ preimage of the boundary of the 
moment polytope. 
With the choice of handles we made above, the intersection of~$L$ with each coordinate-$\drc$ 
consists of two circles lying in the sectors of the coordinate sphere making an 
angle~$\frac{\pi}{2}+\delta$.

We will pick one of these circles and the disk~$u$ it encloses in one of the 
$\frac{\pi}{2}+\delta$-sectors. The area of this disk is equal to the difference 
of the area of one sector and~$2 a(\varepsilon)$:
$$\frac{1}{2}+\frac{\delta}{\pi} - 2 a(\varepsilon).$$
We can now fix~$\delta$ (i.e. take $\delta = 2 \pi a(\varepsilon)$) such that the 
area of~$u$ is equal to~$\frac{1}{2}$.\\

We now compute that the Maslov class of~$u$ is equal to~$1$. 
Let us take~$u$ in the projective line of homogeneous equation~$z_0 = 0$ 
in the homogeneous coordinates $([z_0:z_1],[w_0:w_1])$ of~$\drc \times \drc$. 
In a similar way as in Section~\ref{sec:KSigma-monotonicity}, we trivialize the 
tangent bundle of~$\drc \times \drc$ over~$u$ by considering~$u$ in the 
chart of~$([0:1],[0:1])$.
As before, one can decompose the loop along its boundary in four paths: two paths 
lying on the restriction of the initial Lagrangians to this chart and two paths 
inside the handles.
For the first type of paths, as the restrictions of~$L_0$ and~$L_1$ are linear 
in the chart, they will not contribute to the Maslov class. 
The handle at~$([0:1],[0:1])$ is contributing with a path from~$L_0$ to~$L_1$ 
homotopic to
$$s \in [0,1] \mapsto \left( { 
\begin{array}{cc} e^{i s \left(\frac{\pi}{2}+\delta \right)} & 
\\0 & e^{-i s \left(\frac{\pi}{2}-\delta \right)} \end{array} } \right).$$
The handle at~$([1:0],[0:1])$ is contributing with a path from~$L_0$ to~$L_1$ 
homotopic to
$$s \in [0,1] \mapsto \left( { 
\begin{array}{cc} e^{-i s \left(\frac{\pi}{2}-\delta \right)} & 
\\0 & e^{i s \left(\frac{3\pi}{2}+\delta \right)} \end{array} } \right),$$
so that the Maslov class of~$u$ is the degree of 
$$s \in [0,1] \mapsto e^{-i s 4 \delta} e^{i s (2 \pi + 4 \delta)}$$
that is equal to~$1$.

This is enough for checking the monotonicity of this 
Lagrangian~$K \#  \Sigma_4$ 
in~$\drc \times \drc$ as the relative homology group~$H_2(\drc \times \drc,K \#  \Sigma_4)$ is generated 
by disks with boundary either on~$L_0$ or~$L_1$ (which satisfy the monotonicity condition 
as~$L_0$ and~$L_1$ are monotone) and the disks considered above in the~$\drc$'s.
\end{proof}

\section{Construction of monotone Lagrangian submanifolds using the local model 
along an isotropic circle}
\label{sec:Surgeryinbundle}

\subsection{Case of two copies of a torus intersecting along two circles in~$\drc \times \drc$}

Let~$P = \drc \times \drc$ with the product symplectic form normalized 
as before.

Let us consider the two following Hamiltonian isotopic copies of the real part of~$P$:
$$L_0 = \{ \left. ([x_0:x_1], [u_0:u_1]) \in \drc \times \drc \right|  
x_0, x_1, u_0, u_1 \in \rr \}$$ 
and 
$$L_1 = \{ \left. ([e^{i \frac{\pi}{2}} x_0:x_1], [ u_0:u_1]) \in \drc \times \drc \right|  x_0, x_1, u_0, u_1 \in \rr \}.$$

The two Lagrangians~$L_0$ and~$L_1$ intersect exactly along two isotropic circles:
$$L_0 \cap L_1 = \{[0:1]\} \times \drr \cup \{[1:0]\} \times \drr.$$

Let us study the neighbourhood of~$N = \{[0:1]\} \times \drr$ in~$P$.
The circle~$N$ can be covered by the following two charts of~$\drc \times \drc$:
$$
\begin{array}{cccc}
\phi_0: & U_0 = \{[z_0 : z_1] \mid z_1 \neq 0\} \times \{[w_0 : w_1] \mid w_0 \neq 0\} & \longrightarrow & \cc \oplus \cc \\
        &                        ([z_0:z_1], [w_0 : w_1])                              & \longmapsto     & \left( \frac{z_0}{z_1}, \frac{w_1}{w_0} \right)
\end{array}
$$
and
$$
\begin{array}{cccc}
\phi_1: & U_1 = \{[z_0 : z_1] \mid z_1 \neq 0\} \times \{[w_0 : w_1] \mid w_1 \neq 0\} & \longrightarrow & \cc \oplus \cc \\
        &                        ([z_0:z_1], [w_0 : w_1])                              & \longmapsto     & \left( \frac{z_0}{z_1}, \frac{w_0}{w_1} \right)
\end{array}
$$
so that~$T (\drc \times \drc)_{\mid N}$ can be trivialised along~$N_0 = N \cap U_0$ and~$N_1 = N \cap U_1$
as
$$T (\drc \times \drc)_{\mid N_j} \cong N_j \times (\cc \oplus \cc),$$
where the first summand is just~$\cc = T_{[0:1]} \drc$.
In these trivialisations, 
$$T N_{\mid N_j} \cong N_j \times (\{0\} \oplus \rr)$$
and we have 
$$T N^{\bot}_{\mid N_j} \cong N_j \times (\cc \oplus \rr)$$
so that
$$SN(N, \drc \times \drc)_{\mid N_j} = T N^{\bot} / TN_{\mid N_j} \cong N_j \times  \cc,$$
and the change of trivialisation from~$N_0 \times \cc$ to~$N_1 \times \cc$ in a fibre 
of a point~$([0:1],[c:d])$ of~$N_0 \cap N_1$ is the identity as the fibre at each point 
is~$\cc = T_{[0:1]} \drc$.
We are in the case where the symplectic normal bundle is the trivial complex line bundle over~$N$.

Now, in the trivialisations, the restriction of the initial Lagrangians
to~$SN(N,\drc \times \drc)_{\mid N_j}$ are:
\begin{itemize}
\item for~$L_0$: $N_j \times  \rr$,
\item for~$L_1$: $N_j  \times e^{i \frac{\pi}{2}}  \rr$,
\end{itemize}
with the identity as change of trivialisation so that the restrictions of~$L_0$ and~$L_1$ 
are globally products $N \times  \rr$ and $N  \times e^{i \frac{\pi}{2}}  \rr$ respectively.

As we saw in Section~\ref{sec:isotropicgluing}, this means that the fibrewise surgery with fixed parameter~$\varepsilon$ globalizes to a bundle surgery over~$N$.

The same construction can be done along the other intersection circle~$N' = \{[1:0]\} \times \drr$. 

Let us describe now the choice of surgeries we make to produce a monotone Lagrangian.
In the symplectic normal bundle of~$N$, for any fiber~$\cc$ of a point~$p \in N$ we 
choose the surgery
$$\left \{ { \left. e^{-t} x_0 + e^{t}  e^{i \frac{\pi}{2} } x_0 \right|   
x_0 \in \rr, x_0^2 = \varepsilon_1^2  } \right \}.$$
For the other intersection circle~$N'$, the symplectic normal bundle is again trivial 
and the restriction of~$L_0$ and~$L_1$ 
are  $N' \times  \rr$ and $N'  \times e^{-i \frac{\pi}{2}}  \rr$ respectively, and 
we choose the following handle :
$$\left \{ { \left. e^{-t} x_1 + e^{t}  e^{i \frac{\pi}{2} } (-x_1) \right|   
x_1 \in \rr, x_1^2 = \varepsilon_1^2  } \right \}.$$

The Lagrangian constructed via these surgeries fibers over~$N$ 
(and~$N'$) as both handles and the Lagrangians~$L_0$ and~$L_1$ do, 
the projection being the restriction of the projection of the product onto its 
second factor:
$$pr_2 : \drc \times \drc \longrightarrow \drc.$$
It actually lies in the $\drc$-bundle over~$N$, given by restricting~$pr_2$ to~$\drc \times \drr$.
In a fiber~$\drc$ of this fibration, with our choice of surgery, the restriction of~$L$ 
is one circle in the tennis ball seam shape cutting the fiber into two disks of equal area~$1$.
As through this fibration we see that the Lagrangian is just a product of~$N$ and the 
tennis ball seam, the Maslov class of any of these disks of area~$1$ is the 
Maslov class of the tennis ball seam in the fiber~$\drc$, that is~$2$.

This shows that the Lagrangian we constructed is a torus and that it is monotone.

Unfortunately, this torus is not new, it is Hamiltonian isotopic to the real part 
of~$\drc \times \drc$. For the isotopy we could just take the extension of the exact 
Lagrangian isotopy that isotopes in each fiber of~$pr_2$ the tennis ball seam on the 
real line~$\drr$. We have proved:

\begin{thm}
The Lagrangian bundle surgery construction above produces a monotone Lagrangian 
embedding of a torus $L$ in~$\drc \times \drc$ which is Hamiltonian isotopic to 
the real part of~$\drc \times \drc$ and projects under the moment map to a smooth 
interior approximation of the shaded polytope in Figure~\ref{fig:Toreproduit}.
\end{thm}

\begin{figure}[htbp]
  \begin{center}
   \psfrag{2}{$2$}
   \psfrag{eps1}{$\varepsilon_1^2$}  
   \psfrag{2-eps1}{$2-\varepsilon_1^2$} 
   \includegraphics[height=6cm]{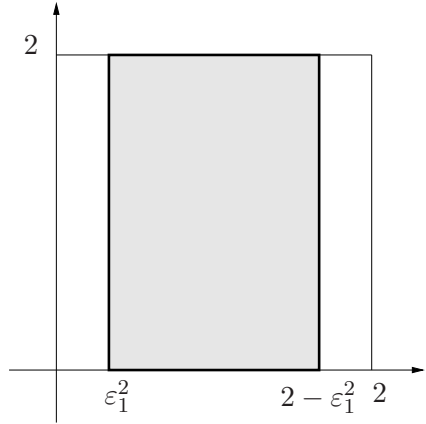}
   \caption{The image of~$L$ under the moment map is contained in and smoothly approximates the shaded polytope.} 
   \label{fig:Toreproduit}
     \end{center}
\end{figure}

\subsection{Recovering the Chekanov--Schlenk torus : Case of two copies of~$\plr$ 
intersecting in a point and along a circle in~$\plc$}
\label{sec:exotictorus}

We detail here how our method produces a torus Hamiltonian isotopic to the model 
torus we started with, i.e. the exotic torus of Chekanov and Schlenk.

\begin{thm}
With a Lagrangian surgery at a point and a Lagrangian surgery along a circle of 
two copies of~$\plr$ one can construct a monotone Lagrangian 
embedding of the torus in~$\plc$ that projects through the moment map to a smooth 
interior approximation of the shaded polytope in Figure~\ref{fig:Thetasurg}. 
\end{thm}

\begin{figure}[htbp]
  \begin{center}
   \psfrag{3}{$3$}
   \psfrag{eps1}{$\varepsilon_1^2$}  
   \psfrag{3-eps1}{$3-\varepsilon_1^2$} 
   \includegraphics[height=6cm]{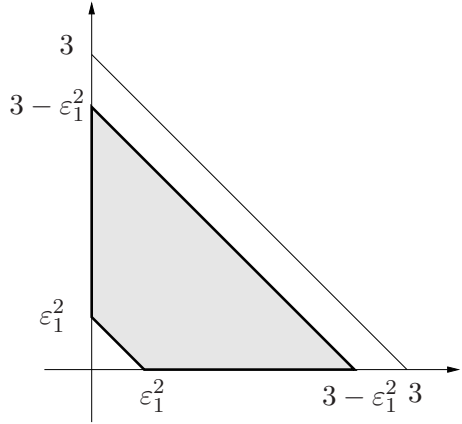}
   \caption{The image of the torus under the moment map is contained in and smoothly approximates the shaded polytope.} 
   \label{fig:Thetasurg}
     \end{center}
\end{figure}

\begin{proof}
Let us consider the following two Hamiltonian isotopic copies of~$\plr$ in~$\plc$:
$$L_0 = \{ \left. [x_0 : x_1 : x_2] \in \plc \right|  x_0, x_1, x_2 \in \rr \}$$
and
$$L_1 = \{ \left. [e^{i \alpha} x_0: e^{i \alpha} x_1: x_2] \in \plc \right|  x_0, x_1, x_2 \in \rr \},$$
for some~$\alpha \in (0, \pi)$.
The two Lagrangians~$L_0$ and~$L_1$ intersect exactly at the point~$[0:0:1]$ and along the 
isotropic circle 
$$N = \{[x_0:x_1:0] \mid (x_0,x_1) \in \rr^2 \setminus \{(0,0)\}\}.$$

We will make a Lagrangian bundle surgery along the isotropic circle 
and a simple Lagrangian surgery at the point~$[0:0:1]$.

Let us first understand the neighbourhood of~$N$ in~$\plc$.
The circle~$N$ can be covered by the following two charts of~$\plc$:
$$
\begin{array}{cccc}
\phi_0: & U_0 = \{[z_0 : z_1 : z_2] \mid z_0 \neq 0\} & \longrightarrow & \cc \oplus \cc \\
        & [z_0:z_1:z_2]                               & \longmapsto     & \left( \frac{z_1}{z_0}, \frac{z_2}{z_0} \right)
\end{array}
$$
and
$$
\begin{array}{cccc}
\phi_1: & U_1 = \{[z_0 : z_1 : z_2] \mid z_1 \neq 0\} & \longrightarrow & \cc \oplus \cc \\
        & [z_0:z_1:z_2]                               & \longmapsto     & \left( \frac{z_0}{z_1}, \frac{z_2}{z_1} \right)
\end{array}
$$
so that~$T \plc_{\mid N}$ can be trivialised along~$N_0 = N \cap U_0$ and~$N_1 = N \cap U_1$
as 
$$T \plc_{\mid N_j} \cong N_j \times (\cc \oplus \cc).$$
In these trivialisations, 
$$T N_{\mid N_j} \cong N_j \times (\rr \oplus \{0\})$$
and we have 
$$T N^{\bot}_{\mid N_j} \cong N_j \times (\rr \oplus \cc)$$
so that
$$SN(N,\plc)_{\mid N_j} = T N^{\bot} / TN_{\mid N_j} \cong N_j \times  \cc,$$
and the change of trivialisation from~$N_0 \times \cc$ to~$N_1 \times \cc$ in a fibre 
over a point~$[a:b:0]$ of~$N_0 \cap N_1$ 
is ~$([a:b:0] , Z) \mapsto  ([a:b:0] , \frac{a}{b} Z)$. 
As the intersection~$N_0 \cap N_1$ retracts onto~$\{[1:1:0], [-1:1:0]\}$, we have only 
two changes of trivialisation to consider: 
in~$[1:1:0]$ it is the identity, and in~$[-1:1:0]$ it is minus the identity.

Now, in the trivialisations, the trace of the initial Lagrangians  
in~$SN(N,\plc)_{\mid N_j}$ are:
\begin{itemize}
\item for~$L_0$: $ N_j \times  \rr$,
\item for~$L_1$: $ N_j \times e^{-i \alpha}  \rr$,
\end{itemize}
with the same change of trivialization as before.

One can then make a Lagrangian surgery in each fibre~$\cc$, and it globalizes to a 
Lagrangian subbundle over~$N$ in the symplectic normal bundle.

We next do a surgery at the transverse intersection point~$[0:0:1]$, so that we get 
an embedded Lagrangian submanifold out of the surgeries on~$L_0$ and~$L_1$. 

We show that for some choices of handles and of~$\alpha$, this Lagrangian 
submanifold is monotone.
Take~$\alpha = \frac{2\pi}{3} + \delta$, for $\delta > 0$ a small parameter to be 
determined later.
We choose the bundle surgery that in each fiber~$\cc$ over a point of~$N$ it 
is parametrized by
$$\left \{ {  e^{-t} x_2  + e^{t}  e^{-i \alpha}  x_2   \left| {  \begin{array}{c} t \in \rr \\ x_2 \in \rr, x_2^2 = \varepsilon_1^2 \end{array} }   \right.  }\right \},$$
followed by a symmetric smoothing.

In the chart at the transverse intersection point~$[0:0:1]$, the two Lagrangians 
are the linear Lagrangian subspaces~$l_0 = \rr^2$ 
and~$l_1 = \{(e^{i \alpha} x_0, e^{i \alpha} x_1) | x_0, x_1 \in \rr\}$ 
and we choose the handle
$$\left \{ { \left. e^{-t} \left( { \begin{array}{c} x_0\\x_1 \end{array} } \right) 
+ e^{t} \left( { \begin{array}{c} e^{i \alpha} x_0\\ e^{i \alpha} x_1 \end{array} } \right) 
\right|   \begin{array}{c} t \in \rr \\x_0^2+x_1^2 = \varepsilon_1^2 \end{array} } \right \}.$$

After surgery, the restriction to the coordinate-$\drc$s of homogeneous equations~$z_0 = 0$ 
(resp. $z_1 = 0$) of the Lagrangian we constructed 
is the union of two circles enclosing disks of area
$$1+\frac{\delta}{2 \pi} - 2 a(\varepsilon).$$
One checks with the same method as before that the Maslov class of these disks 
is equal to~$2$.
Then we choose~$\delta$ such that the area of each of these disks is~$1$.
This is enough to check monotonicity since this Lagrangian is a circle 
subbundle over~$N$ in the normal bundle of~$\{z_2 = 0\}$, and hence a torus. 
Let us denote it by~$\Tetsurg$.
We have checked the monotonicity condition on just one generator 
of the relative second homology group, 
the monotonicity on a second generator is immediate as it can be represented by a disk 
with boundary on~$L_0$ or~$L_1$ for which this condition is satisfied.
\end{proof}

\begin{thm}
The Lagrangian torus~$\Tetsurg$ is Hamiltonian isotopic to the modified Chekanov torus 
in~$\plc$ and consequently also to the Chekanov--Schlenk exotic torus.
\end{thm}

\begin{proof}
We use the strategy of~\cite{Exotic} and prove that 
the torus obtained by surgery is invariant under the same Hamiltonian action (called~$\rCh$ 
in~\cite{Exotic}) as the modified Chekanov torus.

The modified Chekanov torus is invariant under the Hamiltonian circle action~$\rCh$ defined 
on~$\plc$ by applying the following matrix to the homogeneous coordinates:
$$\left ( { \begin{array}{ccc}  
\cos(\theta) & - \sin(\theta) & 0 \\
\sin(\theta) & \cos(\theta)   & 0 \\
      0      &      0         & 1  \end{array}  } \right ),$$
where~$\theta \in \rr / 2 \pi \zz$.

This circle action preserves~$L_0, L_1, N$ and~$[0:0:1]$.

In the chart around~$[0:0:1]$, it restricts on a ball in~$\cc^2$ to the action defined 
by the rotation matrix
$$\left ( { \begin{array}{cc}  \cos(\theta) & - \sin(\theta) \\
      \sin(\theta) & \cos(\theta) \end{array}  } \right ).$$
      
In particular,
\begin{itemize}
\item it preserves the handle of the local surgery at the point~$[0:0:1]$,
\item one can ask the smoothing to be invariant under the action, so that the entire 
surgery is preserved,
\item the action commutes with the homotheties defining the conformal transformation.
\end{itemize}
Moreover, $L_0$ is the orbit under this action restricted to the projective 
line of homogeneous equation~$z_0 = 0$ and even 
to~$\{ [0:x_1:x_2] \in \plc \mid  x_1, x_2 \in \rr, x_1 \geq 0 \}$.
Similarly, $L_1$ is the orbit under this action 
of~$\{[0: e^{i \alpha} x_1 : x_2] \in \plc \mid x_1, x_2 \in \rr, x_1 \geq 0  \}$, 
and the handle is the orbit of one of its branches intersected with~$z_0 = 0$, 
for instance 
of~$\{[ 0: e^{-t} x_1 +  e^{t} e^{i \alpha} x_1 : x_2] \in \plc \mid x_1, x_2 \in \rr, x_1=+ \varepsilon_1 \}$.

One can describe the handle in homogeneous coordinates as
$$\{[x_0 : x_1 : e^{-t} x_2  + e^{t}  e^{-i \alpha} x_2] \in \plc \mid x_0, x_1 \in \rr, x_2^2= \varepsilon_1^2\}$$
and see that it is preserved by the circle action~$\rCh$. In fact, it is the orbit of 
one of its branches in~$z_0 = 0$, for instance
$$\{[0 : x_1 : e^{-t} x_2  + e^{t}  e^{-i \alpha} x_2] \in \plc \mid x_0, x_1 \in \rr, x_2=+ \varepsilon_1\}.$$
We can also take a smoothing that is preserved by the circle action and the orbit 
under the circle action of the smoothing in~$z_0 = 0$.

This shows that the torus~$\Tetsurg$ is the orbit under the circle action~$\rCh$ of 
one of the two circles that constitute its intersection with~$z_0 = 0$. 
Actually, as the third homogeneous coordinate in the handle never vanishes, 
the torus~$\Tetsurg$ lies in the complement of $z_2 = 0$, that is the ball of capacity~$3$ 
around~$[0:0:1]$, and is also in this ball the orbit under the circle action of the circle 
lying in its intersection with the half plane of equations~$z_0 = 0$  and~$\Re e(z_0) \geq 0$. 

As this circle is by construction of area~$1$, one can isotope it inside this half-plane 
to the curve~$\gamma$ used to define~$\tTetCh$.
This isotopy composed with the Hamiltonian circle action gives an exact Lagrangian isotopy 
between the two tori that can be extended to a Hamiltonian isotopy of~$\plc$ as 
in~\cite{Exotic}.
\end{proof}

\bibliographystyle{plain}
\bibliography{Bibliography}

%\address{Center for Mathematical Analysis, Geometry and Dynamical Systems,\\ 
%Instituto Superior T\'ecnico,\\ Universidade de Lisboa\\
%Av. Rovisco Pais, 1049-001 Lisboa, Portugal\\
%\email{mabreu@math.tecnico.ulisboa.pt}}
%
%\address{Centro de Matem\'atica da Universidade do Porto,\\ Departamento de Matem\'atica da FCUP\\
%Rua do Campo Alegre, 687, 4169-007 Porto, Portugal\\
%\email{agnes.gadbled@fc.up.pt}}

\end{document}